\documentclass[12pt]{article}

\usepackage[utf8]{inputenc}

\usepackage[english]{babel}

\usepackage{amsmath,amsfonts,amssymb,amsthm}

\usepackage{amscd}

\usepackage[a4paper,nomarginpar]{geometry}

\usepackage{epsfig,graphicx,color}

\usepackage[all]{xy}

\newtheorem{theorem}{Theorem}
\newtheorem{lemma}[theorem]{Lemma}
\newtheorem{proposition}[theorem]{Proposition}
\newtheorem{corollary}[theorem]{Corollary}

\theoremstyle{definition}
\newtheorem{definition}[theorem]{Definition}
\newtheorem{example}[theorem]{Example}

\newtheorem{remark}[theorem]{Remark}

\newcommand{\N}{\mathbb{N}}  
\newcommand{\Z}{\mathbb{Z}}  
\newcommand{\R}{\mathbb{R}}  
\newcommand{\C}{\mathbb{C}}  
\newcommand{\D}{\mathbb{D}}  
\newcommand{\T}{\mathbb{T}}  

\newcommand{\eps}{\varepsilon} 
\newcommand{\ov} {\overline}   

\newcommand{\shs}{\hspace*{6mm}}

\begin{document}

\title{Expansivity and Shadowing in Linear Dynamics}

\author{Nilson C. Bernardes Jr.\thanks{Partially supported by CNPq.}\,, \
Patricia R. Cirilo \thanks{Partially supported by Fapesp \#2011/11663-5} ,\\
Udayan B. Darji, \
Ali Messaoudi\thanks{Partially supported by CNPq and by Fapesp \#2013/24541-0.} \ and \
Enrique R. Pujals\thanks{Partially supported by CNPq.}}

\date{ }

\maketitle
\begin{abstract}
In the early 1970's Eisenberg and Hedlund investigated relationships between
expansivity and spectrum of operators on Banach spaces. In this paper we
establish relationships between notions of expansivity and hypercyclicity,
supercyclicity, Li-Yorke chaos and shadowing. In the case that the Banach
space is $c_0$ or $\ell_p$ ($1 \leq p < \infty$), we give complete
characterizations of weighted shifts which satisfy various notions of
expansivity. We also establish new relationships between notions of
expansivity and spectrum. Moreover, we study various notions of shadowing
for operators on Banach spaces. In particular, we solve a basic problem in
linear dynamics by proving the existence of nonhyperbolic invertible
operators with the shadowing property. This also contrasts with the expected results for nonlinear dynamics on compact manifolds, illuminating the richness of  dynamics of infinite dimensional linear operators.
\footnote{
2010 {\em Mathematics Subject Classification:} Primary 47A16. Secondary
37B05, 37C50, 54H20.

{\em Keywords:} Expansive, spectrum, hypercyclic, supercyclic, Li-Yorke,
hyperbolic, shadowing, weighted shifts.}
\end{abstract}


\section{Introduction}

\shs
The study of the dynamics of continuous linear operators on
infinite dimensional Banach (or Fr\'echet) spaces has witnessed a great
development during the last three decades and many links between this area
and other areas of mathematics, such as ergodic theory, number theory and
geometry of Banach spaces, have been established. We refer the reader to the
books \cite{FBayEMat09,KGroAPer11} and to the more recent papers
\cite{FBayIRuz15,BerBonMulPer13,BerBonMulPer14,BesMenPerPui16,SGriEMat14,
SGriMRog14}, where many additional references can be found.

\smallskip
On the other hand, the notions of expansivity and shadowing play important
roles in many branches of the area of dynamical systems, including
topological dynamics, differentiable dynamics and ergodic theory;
see \cite{NAokKHir94,SPil99,SPil12,PWal82}, for instance.

\smallskip
Our goal in this paper is to investigate the notions of expansivity and
shadowing in the context of linear dynamics, thereby complementing previous
works by various authors. In particular, we give a class of examples of linear operators exhibiting a shadowing property which are neither hyperbolic nor expansive, however they are chaotic. These types of examples show  the richness of linear dynamics and its difference from finite dimensional nonlinear dynamics, yielding counterintuitive results  to the corresponding ones from finite dimensional smooth  dynamics.

\smallskip
Let us now discuss our results in detail and describe the organization of the article. 

\smallskip
In Section 2, we investigate relationships between various notions of
expansivity and some popular notions in linear dynamics, namely:
hypercyclicity, supercyclicity and Li-Yorke chaos (Definition~\ref{HSLY}).
In particular, we prove that a uniformly expansive operator cannot be
Li-Yorke chaotic and hence it cannot be hypercyclic (Theorem~A),
but we observe that every infinite-dimensional separable Banach space
supports a supercyclic uniformly expansive operator (Remark~\ref{SCUPE}).
On the other hand, we prove that a hyperbolic operator with nontrivial
hyperbolic splitting cannot be supercyclic (Proposition~\ref{HNS}).

\smallskip
In Section 3, we consider weighted shifts. Due to their importance in
operator theory and its applications, the study of the dynamics of weighted
shifts has received special attention from the specialists in linear
dynamics. Many dynamical properties have been extensively studied and,
in some cases, complete characterizations have been obtained. For instance,
Salas \cite{HSal95} characterized hypercyclicity and weak mixing whereas
Costakis and Sambarino \cite{GCosMSam04} characterized mixing for unilateral
and bilateral weighted shifts on $\ell_2(\N)$ and $\ell_2(\Z)$, respectively.
We obtain here complete characterizations of various notions of expansivity
for unilateral and bilateral weighted shifts on the Banach spaces $c_0(A)$
and $\ell_p(A)$ ($1 \leq p < \infty$), where $A = \N$ or $\Z$ (Theorem~B
and Propositions~\ref{BWFS} and \ref{BWBS}). As applications we obtain
examples of hypercyclic positively expansive operators and of supercyclic
uniformly positively expansive operators (Examples~\ref{PEHC} and \ref{UPESC}).

\smallskip
In Section 4, we investigate the relationship between expansivity of an
operator and its spectrum. In particular, we expand earlier results of
Eisenberg and Hedlund \cite{MEis66,MEisJHed70} and Mazur \cite{MMaz00}.
In 1966 Eisenberg \cite{MEis66} proved that if $T$ is an invertible operator
on $\C^n$, then $T$ is expansive if and only if $T$ has no eigenvalue on the
unit circle $\T$. Subsequently, Eisenberg and Hedlund \cite{MEisJHed70}
studied expansive and uniformly expansive operators on Banach spaces.
They showed that if $T$ is uniformly expansive, then $\sigma_a(T)$,
the approximate point spectrum of $T$, does not intersect $\T$.
The converse was shown for invertible operators by Hedlund \cite{JHed71}.
As a corollary, they obtained that invertible hyperbolic operators are
uniformly expansive. Relations between hyperbolicity and the shadowing
property for operators were studied by Ombach \cite{JOmb94} and
Mazur \cite{MMaz00}. In \cite{MMaz00} it was also shown that an invertible
normal operator $T$ on a Hilbert space $H$ is expansive if and only if
$\sigma_p(T^*T)$, the point spectrum of $T^*T$, does not intersect $\T$.
We show that for a uniformly positively expansive operator $T$ on a Banach
space $X$, $\sigma_a(T)$ does not intersect the closed unit disc $\ov{\D}$,
and the converse holds if $T$ is invertible (Theorem~C).
Moreover, we expand Mazur's result by giving a necessary and sufficient
condition for a normal operator to be positively expansive
(Proposition~\ref{NormalPE}). Our techniques also yield a simpler proof
of his result (Theorem~\ref{NormalExp}).

\smallskip
In Section 5, we investigate the notions of shadowing, limit shadowing
and $\ell_p$ shadowing for invertible operators on Banach spaces.
It is well-known that invertible hyperbolic operators have the shadowing
property and that the converse holds for invertible operators on
finite dimensional euclidean spaces \cite{JOmb94} and for invertible normal
operators on Hilbert spaces \cite{MMaz00}.
Moreover, the converse also holds for certain sequences of finite dimensional
operators considered in \cite{SPilSTik10}. This implies that for $C^1$
diffeomorphisms of $m$-dimensional closed smooth manifolds, hyperbolicity
is equivalent to expansivity plus Lipschitz shadowing \cite{SPilSTik10}.
A basic question in linear dynamics is whether the shadowing
property implies hyperbolicity for invertible operators on Banach
(or Hilbert) spaces.
This question appeared explicitly in \cite[Page 148]{MMaz00}, for instance.
In Theorem~D, we answer this question in the negative by proving the
existence of operators with the shadowing property that exhibit several types
of chaotic behaviors (they are simultaneously frequently hypercyclic,
Devaney chaotic, mixing and densely distributionally chaotic) and,
in particular, are not even expansive. Moreover, such type of examples are robust inside the weighted shifted class, meaning that the same properties are share by the perturbed ones.

We also establish a generalization of the aforementioned result from
\cite{MMaz00} (Theorem~\ref{ShadowingOn}) and prove that every expansive
operator with the shadowing property is uniformly expansive
(Proposition~\ref{ESPUE}).

\smallskip
In the final Section 6, we investigate analogous results for operators
which are not necessarily invertible. In Theorem~E, we show that hyperbolic
operators always have the positive shadowing property, the positive limit
shadowing property and the positive $\ell_p$ shadowing property for all
$1 \leq p < \infty$. The converse is not true in general
(Remark~\ref{NotConverse}). We also show that all these notions are
equivalent for compact operators (Theorem~\ref{CompactShadowing}) and that
positive shadowing and hyperbolicity coincide for normal operators
(Theorem~\ref{NormalShadowing2}).


\section{Expansive behavior of operators}

\shs As usual, $\N$ denotes the set of all positive integers and
$\N_0 = \N \cup \{0\}$.
Given a (real or complex) Banach space $X$, $S_X$ denotes the
{\em unit sphere} of $X$, i.e., $S_X = \{x \in X : \|x\| = 1\}$.
Moreover, by an {\em operator} on $X$ we mean a bounded linear map $T$
from $X$ into $X$.
The {\em spectrum} of an operator $T$ on a complex Banach space $X$
is the set
$$
\sigma(T) = \{\lambda \in \C : T - \lambda I \text{ is not invertible}\}.
$$
It is well-known that $\sigma(T)$ is a nonempty compact subset of $\C$.
In the case $T$ is an operator on a real Banach space $X$, we define
the {\em spectrum} of $T$ as the spectrum of its complexification $T_\C$,
that is,
$$
\sigma(T) = \sigma(T_\C).
$$
It is well-known that the spectrum can be divided into three disjoint sets:
the {\em point spectrum} $\sigma_{p}(T)$, the {\em continuous spectrum}
$\sigma_c(T)$ and the {\em residual spectrum} $\sigma_r(T)$.
Recall that $\lambda$ belongs to $ \sigma_p(T)$ if $T - \lambda I$ is not
one-to-one. If $T - \lambda I$ is one-to-one but not onto, then
$\lambda \in \sigma_c(T)$ if $(T- \lambda I) (X)$ is dense in $X$ and
$\lambda \in \sigma_r(T)$ otherwise.
We also recall that the {\em approximate point spectrum} of $T$, denoted by
$\sigma_a(T)$, is the set of all $\lambda \in \C$ for which there is a
sequence $(x_n)$ in $S_X$ with $\|\lambda x_n - Tx_n\| \to 0$ as
$n \to \infty$. It is classical (see \cite{Y}) that
\begin{equation}
\partial \sigma(T) \subset \sigma_a(T) \subset \sigma(T), \ \
\sigma(T) = \sigma_r(T) \cup \sigma_a(T) \ \text{ and } \
\sigma_r(T) \subset \sigma_p(T^*),\label{sig1}
\end{equation}
where $\partial \sigma(T)$ is the boundary of $\sigma (T)$ and $T^*$ is the
adjoint of $T$ acting on the dual space $X^*$ of $X$.

\begin{definition}\label{Def1}
An invertible operator $T$ on a Banach space $X$ is said to be
{\em expansive} ({\em positively expansive}) if for every $z \in S_X$,
there exists $n \in \Z$ ($n \in \N$) such that $\|T^nz\| \geq 2$.
\end{definition}

\begin{definition}\label{Def2}
An invertible operator $T$ on a Banach space $X$ is said to be
{\em uniformly expansive} ({\em uniformly positively expansive}) if
there exists $n \in \N$ such that
$$
z \in S_X \implies \|T^nz\| \geq 2 \text{ or } \|T^{-n}z\| \geq 2
  \ \ \ \ \ (z \in S_X \implies \|T^nz\| \geq 2).
$$
\end{definition}

We remark that for the definitions of positive expansivity and uniform
positive expansivity, $T$ need not be invertible. Also, there is nothing
special about the number $2$ in the above definitions. One can replace $2$
by any number $c > 1$. Moreover, the above definition of expansivity agrees
with the usual definition of expansivity in metric spaces, since it is
equivalent to the existence of a constant $e > 0$ such that, for any pair
$x,y$ of distinct points in $X$, there exists $n \in \Z$ with
$\|T^nx - T^ny\| \geq e$.

\begin{remark}
In the case $T$ is an operator on a real Banach space $X$, the
(uniform) (positive) expansivity of $T$ is equivalent to the corresponding
property for its complexification $T_\C$.
\end{remark}

\begin{definition}\label{Def3}
An operator $T$ on a Banach space $X$ is said to be {\em hyperbolic} if
$$
\sigma(T) \cap \T = \emptyset,
$$
where $\T$ denotes the unit circle in the complex plane $\C$.
\end{definition}

It is classical that $T$ is hyperbolic if and only if there are an equivalent
norm $\|\cdot\|$ on $X$ and a splitting
$$
X = X_s \oplus X_u, \ \ \ \ \ T = T_s \oplus T_u
$$
(the {\em hyperbolic splitting} of $T$), where $X_s$ and $X_u$ are closed
$T$-invariant subspaces of $X$ (the {\em stable} and the {\em unstable
subspaces} for $T$, respectively), $T_s = T|_{X_s}$ is a proper contraction
(i.e., $\|T_s\| < 1$), $T_u = T|_{X_u}$ is invertible and is a proper
dilation (i.e., $\|T_u^{-1}\| < 1$), and the identification of $X$ with
the product $X_s \times X_u$ identifies $\|\cdot\|$ with the max norm
on the product.

\smallskip
It is also known \cite{JHed71} that $T$ is uniformly expansive if and only if
$$
\sigma_a(T) \cap \T = \emptyset.
$$
Hence, every invertible hyperbolic operator is uniformly expansive.

\smallskip
The next result gives simple characterizations of the notions of
expansivity by means of the behaviors of orbits.

\begin{proposition}\label{Characterizations}
Let $T$ be an operator on a Banach space $X$. Then:
\begin{enumerate}
\item [\rm (a)] $T$ is positively expansive $\Leftrightarrow$
      $\sup_{n \in \N} \|T^nx\| = \infty$ for every nonzero $x \in X$.
\item [\rm (b)] $T$ is uniformly positively expansive $\Leftrightarrow$
      $\lim_{n \to \infty} \|T^nx\| = \infty$ uniformly on $S_X$.
\end{enumerate}
If, in addition, $T$ is invertible, then:
\begin{enumerate}
\item [\rm (c)] $T$ is expansive $\Leftrightarrow$
      $\sup_{n \in \Z} \|T^nx\| = \infty$ for every nonzero $x \in X$.
\item [\rm (d)] $T$ is uniformly expansive $\Leftrightarrow$
      $S_X = A \cup B$ where $\lim_{n \to \infty} \|T^nx\| = \infty$
      uniformly on $A$ and $\lim_{n \to \infty} \|T^{-n}x\| = \infty$
      uniformly on $B$.
\end{enumerate}
\end{proposition}

\begin{proof}
(a) and (c) follow easily from the fact, already mentioned, that the constant
$2$ that appears in Definitions \ref{Def1} and \ref{Def2} can be replaced by
any constant $c > 1$.

Let us prove (d). Since the sufficiency of the condition is clear, we have
only to prove its necessity. Suppose that $T$ is uniformly expansive and
let $n \in \N$ be as in Definition~\ref{Def2}. Let
$$
A = \{x \in S_X : \|T^nx\| \geq 2\} \ \ \ \text{ and } \ \ \
B = \{x \in S_X : \|T^{-n}x\| \geq 2\}.
$$
Then, $S_X = A \cup B$. We claim that
$$
\frac{T^nx}{\|T^nx\|} \in A \ \ \text{ whenever } x \in A.
$$
Indeed, if $x \in A$ and $y = \frac{T^nx}{\|T^nx\|} \not\in A$, then
$y \in B$ and so $\|T^{-n}y\| \geq 2$, implying that
$\|x\| \geq 2 \|T^nx\| \geq 4$, a contradiction.
Hence, given $x \in A$, we can define inductively a sequence
$(x_k)_{k \in \N}$ in $A$ by putting $x_1 = x$ and
$x_k = \frac{T^nx_{k-1}}{\|T^nx_{k-1}\|}$ for $k \geq 2$.
It follows from the definition that
$$
x_k = \frac{T^{(k-1)n}x}{\|T^nx_1\| \cdot \ldots \cdot \|T^nx_{k-1}\|} \ \
  \text{ for all } k \in \N.
$$
Since $\|T^nx_k\| \geq 2$ for all $k \in \N$, we obtain
$$
\|T^{kn}x\| \geq 2 \|T^nx_1\| \cdot \ldots \cdot \|T^nx_{k-1}\| \geq 2^k \ \
  \text{ for all } k \in \N.
$$
Let $C = \max_{0 \leq j \leq n-1} \|T^j\| \geq 1$. For each $m \in \N$,
we can write $m = k_m n - j_m$ with unique $k_m \in \N$ and
$j_m \in \{0,\ldots,n-1\}$, and so
$$
\|T^mx\| \geq \frac{2^{k_m}}{\|T^{j_m}\|} \geq \frac{2^{k_m}}{C}\cdot
$$
Since $x \in A$ is arbitrary and $k_m \to \infty$ as $m \to \infty$,
we conclude that $\lim_{m \to \infty} \|T^mx\| = \infty$ uniformly on $A$.
The proof that $\lim_{m \to \infty} \|T^{-m}x\| = \infty$ uniformly on $B$
is analogous.

The proof of (b) is simpler than that of (d) and so we omit it.
\end{proof}

\begin{remark}
The sets $A$ and $B$ in Proposition~\ref{Characterizations}(d) can be chosen
to be disjoint or to be both closed in $S_X$ or to be both open in $S_X$.
\end{remark}

We shall now show that uniformly (positively) expansive operators do not
exhibit chaotic behavior. First, let us recall a few definitions.

\begin{definition}\label{HSLY}
An operator $T$ on a Banach space $X$ is said to be {\em Li-Yorke chaotic}
if it has an uncountable {\em scrambled set} $U$, i.e., for all $x, y \in U$
with $x \neq y$, we have that
$$
\liminf_{n \to \infty} \|T^nx - T^ny\| = 0 \ \ \text{ and } \ \
\limsup_{n \to \infty} \|T^nx - T^ny\| > 0.
$$
We say that $T$ is {\em hypercyclic} if it has a dense orbit, i.e.,
$$
\{T^nx: n \geq 0\}
$$
is dense in $X$ for some $x \in X$. Finally, $T$ is {\em supercyclic}
if there exists $x \in X$ whose projective orbit
$$
\{\lambda T^nx : n \geq 0, \lambda \text{ scalar}\}
$$
is dense in $X$.
\end{definition}

\medskip
\noindent {\bf Theorem A.} {\it A uniformly (positively) expansive operator
on a Banach space cannot be Li-Yorke chaotic. In particular, it cannot be
hypercyclic.}

\medskip
\begin{proof}
Let us consider the case of a uniformly expansive (necessarily invertible)
operator $T$ on a Banach space $X$ (the case of a uniformly positively
expansive (not necessarily invertible) operator is simpler). Write
$S_X = A \cup B$ as in Proposition~\ref{Characterizations}(d).
It was proved in \cite[Theorem 5]{BerBonMarPer11} that $T$ is Li-Yorke
chaotic if and only if $T$ admits an {\em irregular vector}, that is,
a vector $x \in X$ such that
$$
\inf_{n \in \N} \|T^nx\| = 0 \ \ \ \text{ and } \ \ \
\sup_{n \in \N} \|T^nx\| = \infty.
$$
Suppose that $T$ is Li-Yorke chaotic and let $y \in S_X$ be an irregular
vector for $T$. We must have
$$
\frac{T^ky}{\|T^ky\|} \in B \ \ \text{ for all } k \in \N.
$$
Indeed, if $\frac{T^ky}{\|T^ky\|} \in A$ for some $k \in \N$, then
$\lim_{n \to \infty} \big\|T^n\big(\frac{T^ky}{\|T^ky\|}\big)\big\| = \infty$,
which implies that $\lim_{n \to \infty} \|T^ny\| = \infty$ and contradicts
the fact that $y$ is an irregular vector for $T$.
Since $\lim_{n \to \infty} \|T^{-n}x\| = \infty$ uniformly on $B$,
there exists $n_0 \in \N$ such that
\begin{equation}
\|T^{-n}x\| \geq 2 \ \ \text{ whenever } x \in B \text{ and } n \geq n_0.
\label{eq}
\end{equation}
Since $y$ is an irregular vector for $T$, we can choose $k_0 \geq n_0$
such that $\|T^{k_0}y\| \geq 1$. Now, by choosing $n = k_0 \geq n_0$ and
$x = \frac{T^{k_0}y}{\|T^{k_0}y\|} \in B$ in (\ref{eq}), we obtain
$\|y\| \geq 2\|T^{k_0}y\| \geq 2$. This contradiction proves the theorem.
\end{proof}

\begin{remark}
The fact that a uniformly expansive operator $T$ on a Banach space $X$ cannot
be hypercyclic can be seen by means of a spectral argument. Indeed, if $T$ is
hypercyclic, then its spectrum $\sigma(T)$ intersects the unit circle and the
point spectrum $\sigma_p(T^*)$ of the adjoint operator $T^*$ is empty
(see \cite{FBayEMat09}). Since $\sigma_r(T) \subset \sigma_p(T^*)$, we deduce
that $\sigma_r(T)$ is empty. Since $\sigma(T) = \sigma_r(T) \cup \sigma_a(T)$,
we have that $\sigma_a(T)$ intersects the unit circle and thus $T$ is not
uniformly expansive.
\end{remark}

On the other hand, there exist supercyclic uniformly (positively) expansive
operators, as we shall see in the remark below. A simple concrete example
of such an operator on the Hilbert space $\ell_2$ will be given in
Example~\ref{UPESC}.

\begin{remark}\label{SCUPE}
Every infinite-dimensional separable Banach space admits an invertible
operator which is uniformly positively expansive and supercyclic.

\smallskip
Indeed, it is well-known that every infinite-dimensional separable Banach
space $X$ supports a hypercyclic invertible operator $S$
(see \cite[Section~8.2]{KGroAPer11}).
Since any nonzero scalar multiple of a supercyclic operator is a
supercyclic operator,
$$
T = 2 \|S^{-1}\| S
$$
is a supercyclic operator on $X$. Moreover, since $\|T^{-1}\| = \frac{1}{2}
< 1$, $T$ is a proper dilation. In particular, $T$ is uniformly positively
expansive.
\end{remark}

\begin{remark}
A positively expansive operator can be Li-Yorke chaotic. For example,
Beau\-zamy \cite{BBea88} and Pr\v{a}jitur\v{a} \cite{GPra09}
constructed examples of {\em completely irregular operators} on the
Hilbert space $\ell_2$. These are operators with the property that
every nonzero vector is irregular. It follows from
Proposition~\ref{Characterizations}(a) and \cite[Theorem 34]{BerBonMulPer14}
that every completely irregular operator on a Banach space is simultaneously
positively expansive and generically Li-Yorke chaotic.
Also, Read~\cite{CRea88} constructed an operator $T$ on $\ell_1$
with all nonzero vectors hypercyclic. This operator is simultaneously
positively expansive, generically Li-Yorke chaotic and hypercyclic.
Moreover, we shall see later an example of an invertible operator on the
Hilbert space $\ell_2$ which is positively expansive (hence expansive) and
hypercyclic (Example~\ref{PEHC}).
\end{remark}

\smallskip
As we saw in Remark~\ref{SCUPE}, a uniformly expansive operator can be
supercyclic. However, this is not the case for hyperbolic operators with
nontrivial hyperbolic splittings.

\begin{proposition}\label{HNS}
If $T$ is a hyperbolic operator with nontrivial hyperbolic splitting,
then $T$ is not supercyclic.
\end{proposition}

\begin{proof}
By hypothesis, there is a splitting
$$
X = X_s \oplus X_u, \ \ \ \ \ T = T_s \oplus T_u,
$$
as above, with $X_s \neq \{0\}$ and $X_u \neq \{0\}$.
By renorming $X$ we may assume that $\|T_s\| < 1$ and $\|T_u^{-1}\| < 1$.
Each vector $x \in X$ has a unique decomposition $x = x_s + x_u$
with $x_s \in X_s$ and $x_u \in X_u$.
Moreover, by the open mapping theorem, the mapping
$x \in X \mapsto (x_s,x_u) \in X_s \times X_u$ is an isomorphism.
Suppose that $T$ admits a supercyclic vector $y \in X$.
It must be true that $y_s \neq 0$ and $y_u \neq 0$.
Since $\|T^ny_s\| \to 0$ and $\|T^ny_u\| \to \infty$, there exists
$n_0 \in \N$ such that
$$
\|T^ny_u\| \geq 2\|T^ny_s\| \ \ \text{ whenever } n \geq n_0.
$$
On one hand, the set
$$
D = \{\lambda T^ny : \lambda \text{ is a scalar and } n \geq n_0\}
$$
is dense in $X$. But on the other hand, each element $z = \lambda T^ny \in D$
has decomposition $z = z_s + z_u = \lambda T^ny_s + \lambda T^ny_u$
satisfying $\|z_u\| \geq 2 \|z_s\|$, and so $D$ cannot be dense in $X$.
This contradiction proves the proposition.
\end{proof}

\begin{remark}
Another way to see the last result in the case of complex scalars comes from
the fact that if $T$ is supercyclic, then there exists $R \geq 0$ such that
each connected component of the spectrum of $T$ intersects the (possibly
degenerate) circle $\{z \in \C : |z| = R\}$ (see \cite{FBayEMat09}).
This is impossible if $T$ is a hyperbolic operator with nontrivial
hyperbolic splitting, since the unit circle separates at least two connected
components of $\sigma(T)$.
\end{remark}


\section{Expansive weighted shifts}

\shs In this section we characterize the various notions of expansivity
for weighted shifts by looking at their weights.

\smallskip
For each real number $p \in [1,\infty)$, we denote by $\ell_p(\Z)$ the
Banach space of all sequences $x = (x_n)_{n \in \Z}$ of scalars such that
$\sum_{n \in \Z} |x_n|^p < \infty$, endowed with the norm
$$\|x\| = \Bigg(\sum_{n \in \Z} |x_n|^p\Bigg)^\frac{1}{p}.$$
In particular, $\ell_2(\Z)$ is a Hilbert space with respect to the inner
product
$$\langle x,y \rangle = \sum_{n \in \Z} x_n\ov{y_n}.$$
Moreover, $c_0(\Z)$ denotes the Banach space of all sequences
$x = (x_n)_{n \in \Z}$ of scalars such that $\lim_{n \to \pm \infty} x_n = 0$,
endowed with the norm
$$\|x\| = \sup_{n \in \Z} |x_n|.$$
The Banach spaces $\ell_p(\N)$ ($1 \leq p < \infty$) and $c_0(\N)$ are
defined analogously.

\smallskip
If $X = \ell_p(\Z)$ $(1 \leq p < \infty)$ or $X = c_0(\Z)$, then
$F_w : X \to X$ ($B_w : X \to X$) denotes the {\em bilateral weighted forward}
({\em backward}) {\em shift} on $X$ given by
$$
F_w\big((x_n)_{n \in \Z}\big) = (w_{n-1}x_{n-1})_{n \in \Z} \ \ \ \
\big(B_w\big((x_n)_{n \in \Z}\big) = (w_{n+1}x_{n+1})_{n \in \Z}\big),
$$
where $w = (w_n)_{n \in \Z}$ is a bounded sequence of scalars, called a
{\em weight sequence}. Recall that
$$
F_w \ (B_w) \text{ is invertible } \
  \Longleftrightarrow \ \ \inf_{n \in \Z} |w_n| > 0.
$$
In the case $X = \ell_p(\N)$ $(1 \leq p < \infty)$ or $X = c_0(\N)$, we also
denote by $F_w : X \to X$ ($B_w : X \to X$) the {\em unilateral weighted
forward} ({\em backward}) {\em shift} on $X$ with {\em weight sequence}
$w = (w_n)_{n \in \N}$, which is defined by
$$
F_w\big((x_1,x_2,\ldots)\big) = (0,w_1x_1,w_2x_2,\ldots) \ \ \ \
\big(B_w\big((x_1,x_2,\ldots)\big) = (w_2x_2,w_3x_3,\ldots)\big).
$$
We remark that in this case the weight sequence $w$ is also assumed to be
bounded.

\smallskip
We begin by characterizing (uniform) expansivity for invertible bilateral
weighted forward shifts. For this purpose, we will need the following fact.

\begin{lemma}\label{Lemma1}
If $\{I,J\}$ is a nontrivial partition of $\Z$ (that is, $I \cup J = \Z$,
$I \cap J = \emptyset$, $I \neq \emptyset$ and $J \neq \emptyset$),
$\varphi : \Z \to [0,\infty)$ is a map,
$$
\lim_{n \to \infty} \big[\inf_{k \in I}
  \big(\varphi(k) \cdot\ldots\cdot \varphi(k+n-1)\big)\big] > 1
$$
and
$$
\lim_{n \to \infty} \big[\sup_{k \in J}
  \big(\varphi(k-n) \cdot\ldots\cdot \varphi(k-1)\big)\big] < 1,
$$
then there exist $i,j \in \Z$ such that
$$
(-\infty,j] \cap \Z \subset J \ \ \ \text{ and } \ \ \
[i,\infty) \cap \Z \subset I.
$$
\end{lemma}

\begin{proof}
By hypothesis, there exists $n_0 \in \N$ such that
$$
\varphi(k) \cdot\ldots\cdot \varphi(k+n-1) > 1 \ \ \text{ for all } k \in I
$$
and
$$
\varphi(k-n) \cdot\ldots\cdot \varphi(k-1) < 1 \ \ \text{ for all } k \in J,
$$
whenever $n \geq n_0$. We claim that
\begin{equation}
k \in I \ \Rightarrow \ k+n \in I \ \text{ for all } n \geq n_0. \label{eq1}
\end{equation}
Indeed, suppose that $k \in I$ but $k+n \in J$ for a certain $n \geq n_0$.
Then,
$$
\varphi(k) \cdot\ldots\cdot \varphi(k+n-1) =
\varphi((k+n)-n) \cdot\ldots\cdot \varphi((k+n)-1)
$$
is simultaneously $> 1$ and $< 1$, because $k \in I$, $k+n \in J$ and
$n \geq n_0$. This contradiction proves (\ref{eq1}).
Analogously, we have that
\begin{equation}
k \in J \ \Rightarrow \ k-n \in J \ \text{ for all } n \geq n_0. \label{eq2}
\end{equation}
Since $I \neq \emptyset$ and $J \neq \emptyset$, it is clear that
(\ref{eq1}) and (\ref{eq2}) imply the existence of $i,j \in \Z$ with
the desired properties.
\end{proof}

\medskip
\noindent {\bf Theorem B.}
{\it Let $X = \ell_p(\Z)$ $(1 \leq p < \infty)$ or $X = c_0(\Z)$, and consider
a weight sequence $w = (w_n)_{n \in \Z}$ with $\inf_{n \in \Z} |w_n| > 0$.
\begin{enumerate}
\item [\rm (a)] The following assertions are equivalent:
      \begin{enumerate}
      \item [\rm (i)]   $F_w : X \to X$ is expansive;
      \item [\rm (ii)]  $F_w : X \to X$ or $F_w^{-1} : X \to X$ is positively
            expansive;
      \item [\rm (iii)] $\displaystyle
            \sup_{n \in \N} |w_1 \cdot\ldots\cdot w_n| = \infty$ or
            $\displaystyle
            \sup_{n \in \N} |w_{-n} \cdot\ldots\cdot w_{-1}|^{-1} = \infty$.
      \end{enumerate}
\item [\rm (b)] The following assertions are equivalent:
      \begin{enumerate}
      \item [\rm (i)]   $F_w : X \to X$ is uniformly expansive;
      \item [\rm (ii)] One of the following conditions holds:
            \begin{itemize}
            \item $\displaystyle \lim_{n \to \infty} \big(\inf_{k \in \Z}
                  |w_k \cdot\ldots\cdot w_{k+n-1}|\big) = \infty$, or
            \item $\displaystyle \lim_{n \to \infty} \big(\inf_{k \in \Z}
                  |w_{k-n} \cdot\ldots\cdot w_{k-1}|^{-1}\big) = \infty$, or
            \item $\displaystyle \lim_{n \to \infty} \big(\inf_{k \in \N}
                  |w_k \cdot\ldots\cdot w_{k+n-1}|\big) = \infty$ and
                  $\displaystyle \lim_{n \to \infty} \big(\inf_{k \in -\N}
                  |w_{k-n} \cdot\ldots\cdot w_{k-1}|^{-1}\big) = \infty$.
            \end{itemize}
      \end{enumerate}
\end{enumerate}}

\smallskip
\begin{proof}
Let $e_j$, $j \in \Z$, denote the canonical unit vectors in $X$.

\smallskip
\noindent (a): If $F_w$ is expansive, then
Proposition~\ref{Characterizations}(c) implies that
$$
\sup_{n \in \N} \|F_w^n(e_1)\| = \infty \ \ \ \text{ or } \ \ \
\sup_{n \in \N} \|F_w^{-n}(e_1)\| = \infty.
$$
The first equality means that $\sup_{n \in \N} |w_1 \cdot\ldots\cdot
w_n| = \infty$, whereas the second one means that $\sup_{n \in \N}
|w_{-n+1} \cdot\ldots\cdot w_0|^{-1} = \infty$, which is clearly
equivalent to $\sup_{n \in \N} |w_{-n} \cdot\ldots\cdot w_{-1}|^{-1}
= \infty$. This shows that (i) implies (iii).
Now, assume that $\sup_{n \in \N} |w_1 \cdot\ldots\cdot w_n| = \infty$.
Let $x = (x_j)_{j \in \Z}$ be any nonzero vector in $X$ and choose
$k \in \Z$ such that $x_k \neq 0$. Then,
\begin{align*}
\sup_{n \in \N} \|F_w^n(x)\|
  &\geq \sup_{n \in \N} \big|(w_k \cdot\ldots\cdot w_{k+n-1}) x_k\big|\\
  &= \frac{|x_k| \prod_{j=k}^0 |w_j|}{\prod_{j=1}^{k-1} |w_j|}
    \sup_{n \in \N} |w_1 \cdot\ldots\cdot w_{k+n-1}| = \infty,
\end{align*}
where a product over an empty set of indices has value $1$, by definition.
Hence, by Proposition~\ref{Characterizations}(a), $F_w$ is positively
expansive. Analogously, the relation
$\sup_{n \in \N} |w_{-n} \cdot\ldots\cdot w_{-1}|^{-1} = \infty$
implies that $F_w^{-1}$ is positively expansive.
Thus, (iii) implies (ii). Finally, it is trivial that (ii) implies (i).

\smallskip
\noindent (b): Suppose that $F_w$ is uniformly expansive.
By Proposition~\ref{Characterizations}(d), there is a partition $\{A,B\}$
of $S_X$ such that
$$
\lim_{n \to \infty} c_n = \infty \ \ \ \text{ and } \ \ \
\lim_{n \to \infty} d_n = \infty,
$$
where
$$
c_n = \inf_{x \in A} \|F_w^n(x)\| \ \ \ \text{ and } \ \ \
d_n = \inf_{x \in B} \|F_w^{-n}(x)\| \ \ \ \ (n \in \N).
$$
We remark that an infimum over an empty set of indices has value $\infty$,
by definition. Let
$$
I = \{k \in \Z : e_k \in A\} \ \ \ \text{ and } \ \ \
J = \{k \in \Z : e_k \in B\}.
$$
Then $\{I,J\}$ is a partition of $\Z$. Since, for all $n \in \N$,
$$
\inf_{k \in I} |w_k \cdot\ldots\cdot w_{k+n-1}|
  = \inf_{k \in I} \|F_w^n(e_k)\| \geq c_n
$$
and
$$
\inf_{k \in J} |w_{k-n} \cdot\ldots\cdot w_{k-1}|^{-1}
  = \inf_{k \in J} \|F_w^{-n}(e_k)\| \geq d_n,
$$
we conclude that
\begin{equation}
\lim_{n \to \infty} \big(\inf_{k \in I}
   |w_k \cdot\ldots\cdot w_{k+n-1}|\big) = \infty \ \ \text{ and } \ \
\lim_{n \to \infty} \big(\inf_{k \in J}
   |w_{k-n} \cdot\ldots\cdot w_{k-1}|^{-1}\big) = \infty. \label{Eq1}
\end{equation}
Thus, $J = \emptyset$ gives the first possibility in (ii) while
$I = \emptyset$ gives the second one.
Assume that $I \neq \emptyset$ and $J \neq \emptyset$.
By Lemma~\ref{Lemma1}, there exist $i,j \in \Z$ such that
\begin{equation}
(-\infty,j] \cap \Z \subset J \ \ \ \text{ and } \ \ \
[i,\infty) \cap \Z \subset I. \label{Eq2}
\end{equation}
Since $w_k \neq 0$ for all $k \in \Z$, it is easy to see that (\ref{Eq1})
and (\ref{Eq2}) imply the third possibility in~(ii).

Conversely, assume that (ii) holds.
Let $I = \Z$ and $J = \emptyset$, or $I = \emptyset$ and $J = \Z$, or
$I = \N$ and $J = -\N_0$, depending on whether the first, the second,
or the third possibility in (ii) holds, respectively. Then, in any case,
(\ref{Eq1}) holds. Let $n \in \N$ be such that
$$
\inf_{k \in I} |w_k \cdot\ldots\cdot w_{k+n-1}| \geq 4 \ \ \text{ and } \ \
\inf_{k \in J} |w_{k-n} \cdot\ldots\cdot w_{k-1}|^{-1} \geq 4.
$$
Given $x = (x_k)_{k \in \Z} \in S_X$, we can write $x = a + b$ where
$a = (a_k)_{k \in \Z}$ and $b = (b_k)_{k \in \Z}$ satisfy
$a_k = 0$ whenever $k \in J$ and $b_k = 0$ whenever $k \in I$.
Since $1 = \|x\| \leq \|a\| + \|b\|$, we have that $\|a\| \geq \frac{1}{2}$
or $\|b\| \geq \frac{1}{2}$. If $\|a\| \geq \frac{1}{2}$ then
$$
\|F_w^n(x)\| \geq \|F_w^n(a)\|
  = \big\| \big((w_k \cdot\ldots\cdot w_{k+n-1}) a_k\big)_{k \in \Z} \big\|
  \geq 4\|a\| \geq 2,
$$
and if $\|b\| \geq \frac{1}{2}$ then
$$
\|F_w^{-n}(x)\| \geq \|F_w^{-n}(b)\|
  = \big\| \big((w_{k-n} \cdot\ldots\cdot w_{k-1})^{-1} b_k\big)_{k \in \Z} \big\|
  \geq 4\|b\| \geq 2.
$$
Hence, by definition, $F_w$ is uniformly expansive.
\end{proof}

By using analogous (but simpler) arguments, we can establish the following
characterizations of (uniform) positive expansivity for weighted forward
shifts.

\begin{proposition}\label{BWFS}
Let $A = \N$ or $A = \Z$, let $X = \ell_p(A)$ $(1 \leq p < \infty)$ or
$X = c_0(A)$, and consider a weight sequence $w = (w_n)_{n \in A}$.
\begin{enumerate}
\item [\rm (a)] The following assertions are equivalent:
      \begin{enumerate}
      \item [\rm (i)]  $F_w : X \to X$ is positively expansive;
      \item [\rm (ii)] $\displaystyle \sup_{n\in\N} |w_1 \cdot\ldots\cdot w_n|
            = \infty$ and $w_j \neq 0$ for all $j \in A $.
      \end{enumerate}
\item [\rm (b)] The following assertions are equivalent:
      \begin{enumerate}
      \item [\rm (i)]   $F_w : X \to X$ is uniformly positively expansive;
      \item [\rm (ii)]  $\displaystyle \sup_{n \in \N} \big(\inf_{k \in A}
            |w_k \cdot\ldots\cdot w_{k+n-1}|\big) = \infty$;
      \item [\rm (iii)] $\displaystyle \lim_{n \to \infty} \big(\inf_{k \in A}
            |w_k \cdot\ldots\cdot w_{k+n-1}|\big) = \infty$.
      \end{enumerate}
\end{enumerate}
\end{proposition}

It is clear that a unilateral weighted backward shift cannot be positively
expansive, but for bilateral weighted backward shifts we have the following
characterizations.

\begin{proposition}\label{BWBS}
Let $X = \ell_p(\Z)$ $(1 \leq p < \infty)$ or $X = c_0(\Z)$, and consider
a weight sequence $w = (w_n)_{n \in \Z}$.
\begin{enumerate}
\item [\rm (a)] The following assertions are equivalent:
      \begin{enumerate}
      \item [\rm (i)]  $B_w : X \to X$ is positively expansive;
      \item [\rm (ii)] $\displaystyle
            \sup_{n \in \N} |w_{-n} \cdot\ldots\cdot w_{-1}| = \infty$
            and $w_j \neq 0$ for all $j \geq 0$.
      \end{enumerate}
\item [\rm (b)] The following assertions are equivalent:
      \begin{enumerate}
      \item [\rm (i)]   $B_w : X \to X$ is uniformly positively expansive;
      \item [\rm (ii)]  $\displaystyle \sup_{n \in \N} \big(\inf_{k \in \Z}
            |w_{k-n+1} \cdot\ldots\cdot w_k|\big) = \infty$;
      \item [\rm (iii)] $\displaystyle \lim_{n \to \infty} \big(\inf_{k \in \Z}
            |w_{k-n+1} \cdot\ldots\cdot w_k|\big) = \infty$.
      \end{enumerate}
\end{enumerate}
\end{proposition}

\begin{remark}\label{PE-E}
If $T$ is an invertible operator on a Banach space $X$, it is clear that
$$
T \text{ or } T^{-1} \text{ positively expansive } \ \Rightarrow \
T \text{ expansive}.
$$
We saw in Theorem~B(a) that the converse holds for the operators
$F_w$ on the spaces $\ell_p(\Z)$ ($1 \leq p < \infty$) or $c_0(\Z)$.
Of course, the converse is not true in general. For instance, if $T$ is any
invertible hyperbolic operator with nontrivial hyperbolic splitting, then
$T$ is uniformly expansive, but neither $T$ nor $T^{-1}$ is positively
expansive.
\end{remark}

\begin{remark}
(a) It follows from the equivalence (ii) $\Leftrightarrow$ (iii) in
Proposition~\ref{BWFS}(b) that all limits in Theorem~B(b)(ii)
can be replaced by the supremum over $n \in \N$.

\smallskip
\noindent (b)
The first possibility in Theorem~B(b)(ii) means that $F_w$ is
uniformly positively expansive (by Proposition~\ref{BWFS}(b)), whereas the
second one means that $F_w^{-1}$ is uniformly positively expansive
(by  Proposition~\ref{BWBS}(b)).
However, the third possibility can indeed happen, as can be seen by
choosing $w = (\ldots,\frac{1}{2},\frac{1}{2},\frac{1}{2},2,2,2,\ldots)$.
This shows that $F_w$ can be uniformly expansive without $F_w$ or $F_w^{-1}$
being uniformly positively expansive, in contrast to what happens in the case
of expansivity (see Theorem~B(a)).
\end{remark}

Let us now see an example of an invertible operator on the Hilbert space
$\ell_2(\Z)$ which is positively expansive and hypercyclic.

\begin{example}\label{PEHC}
Fix a real number $t > 1$ and consider the weight sequence
$w = (w_n)_{n \in \Z}$ given by
$$
w_n = t \ \ \text{ for all } n \geq 0
$$
and
$$
(w_{-1},w_{-2},w_{-3},\ldots) = (t,\frac{1}{t},\frac{1}{t},t,t,t,t,
\frac{1}{t},\ldots,\frac{1}{t},t,\ldots,t,\ldots),
$$
where the successive blocks of $t$'s and $\frac{1}{t}$'s have lengths
$2^0,2^1,2^2,\ldots$. Let
$$
m_k = 2^0 + 2^1 + \cdots + 2^{2k-1} \ \ \text{ and } \ \
n_k = 2^0 + 2^1 + \cdots + 2^{2k} \ \ \ (k \in \N).
$$
A simple induction argument shows that
$$
w_{-m_k} \cdot\ldots\cdot w_{-1} \leq \frac{1}{t^k} \ \ \text{ and } \ \
w_{-n_k} \cdot\ldots\cdot w_{-1} \geq t^k \ \ \text{ for all } k \in \N.
$$
In particular, $\sup_{n \in \N} (w_{-n} \cdot\ldots\cdot w_{-1}) = \infty$.
Hence, by Proposition~\ref{BWBS}(a), the bilateral weighted backward shift
$$
B_w : \ell_2(\Z) \to \ell_2(\Z)
$$
is positively expansive. Since $\inf_{n \in \Z} w_n > 0$, $B_w$ is invertible.
Hence, $B_w$ is also expansive. By \cite[Corollary 1.39]{FBayEMat09},
$B_w$ is hypercyclic if and only if, for any $q \in \N$,
$$
\liminf_{n \to \infty} \max\big\{(w_1 \cdot \ldots \cdot w_{n+q})^{-1},
(w_0 \cdot \ldots \cdot w_{-n+q+1})\big\} = 0.
$$
But this condition follows from the fact that
$$
\max\big\{(w_1 \cdot \ldots \cdot w_{(m_k+q+1)+q})^{-1},
          (w_0 \cdot \ldots \cdot w_{-(m_k+q+1)+q+1})\big\}
  \leq \frac{1}{t^{k-1}} \ \ \text{ for all } k \in \N.
$$
Thus, the operator $B_w$ is also hypercyclic.
\end{example}

\begin{remark}\label{HyperbolicShifts}
(a) Let $X = \ell_p(\Z)$ $(1 \leq p < \infty)$ or $X = c_0(\Z)$, and consider
a weight sequence $w = (w_n)_{n \in \Z}$ with $\inf_{n \in \Z} |w_n| > 0$.
It is known (see \cite{pos}) that the spectrum of the invertible bilateral
weighted forward shift $F_w : X \to X$ is the annulus
$\{\lambda \in \C : \frac{1}{r(F_w^{-1})} \leq |\lambda| \leq r(F_w)\}$,
where $r(S) = \lim_{n \to \infty} \|S^n\|^\frac{1}{n}$
denotes the spectral radius of the operator $S : X \to X$.
Since
$$
\|F_w^n\| = \sup_{k \in \Z} |w_k \cdot\ldots\cdot w_{k+n-1}|
\ \ \text{ and } \ \
\|F_w^{-n}\| = \sup_{k \in \Z} |w_k \cdot\ldots\cdot w_{k+n-1}|^{-1},
$$
we deduce that the following assertions are equivalent:
\begin{enumerate}
\item [\rm (i)]   $F_w$ is hyperbolic;
\item [\rm (ii)]  $\sigma(F_w) \subset \D$ or $\sigma(F_w^{-1}) \subset \D$;
\item [\rm (iii)] $\displaystyle \lim_{n \to \infty} \sup_{k \in \Z}
                  |w_k \cdot \ldots \cdot w_{k+n-1}|^{\frac{1}{n}} < 1$ or
                  $\displaystyle \lim_{n \to \infty} \sup_{k \in \Z}
                  |w_k \cdot \ldots \cdot w_{k+n-1}|^{-\frac{1}{n}} < 1$.
\end{enumerate}

\noindent (b) Let $A = \N$ or $A = \Z$,
let $X = \ell_p(A)$ $(1 \leq p < \infty)$ or $X = c_0(A)$,
and consider a weight sequence $w = (w_n)_{n \in A}$.
Let $T$ be either the weighted forward shift $F_w : X \to X$ or the
weighted backward shift $B_w : X \to X$. Assume that $T$ is not invertible
(this is automatically the case if $A = \N$).
Since $\sigma(T)$ is equal to the disc
$\{\lambda \in \C :  |\lambda| \leq r(T)\}$ (see \cite{pos}),
we deduce that the following assertions are equivalent:
\begin{enumerate}
\item [\rm (i)]   $T$ is hyperbolic;
\item [\rm (ii)]  $\sigma(T) \subset \D$;
\item [\rm (iii)] $\displaystyle \lim_{n \to \infty} \sup_{k \in A}
                  |w_k \cdot \ldots \cdot w_{k+n-1}|^{\frac{1}{n}} < 1$.
\end{enumerate}
\end{remark}

\begin{remark}
It follows immediately from the definitions that an invertible operator $T$
is expansive (uniformly expansive, hyperbolic) if and only if so is its
inverse operator $T^{-1}$. Hence, the study of these notions for invertible
bilateral weighted backward shifts can be reduced to the corresponding case
of forward shifts (see Theorem~B and Remark~\ref{HyperbolicShifts}(a)).
\end{remark}

\begin{remark}
(a) As mentioned before, it was proved in \cite{MEisJHed70} that every
invertible hyperbolic operator is uniformly expansive. Examples of
uniformly expansive nonhyperbolic operators were also obtained in
\cite{MEisJHed70}. We observe that such examples can be easily obtained
by using the characterizations given in Theorem~B(b) and
Remark~\ref{HyperbolicShifts}(a).

\smallskip
\noindent (b) In the case of noninvertible operators, we observe that
there is no relation between hyperbolicity and uniform positive expansivity
in general. For instance, it follows from Proposition~\ref{BWFS}(a) and
Remark~\ref{HyperbolicShifts}(b) that in the class of unilateral weighted
forward shifts on $\ell_p(\N)$ ($1 \leq p < \infty$) or on $c_0(\N)$,
the set of hyperbolic shifts is disjoint from the set of positively
expansive shifts.
\end{remark}

Let us now see a concrete example of an invertible operator on the Hilbert
space $\ell_2(\Z)$ which is uniformly positively expansive and supercyclic.

\begin{example}\label{UPESC}
Fix real numbers $\alpha > \beta > 1$ and consider the weight sequence
$$
w = (w_n)_{n \in \Z} = (\ldots,\beta,\beta,\beta,\alpha,\alpha,\alpha,\ldots),
$$
where the first $\alpha$ appears at position $1$. Consider the bilateral
weighted backward shift
$$
B_w : \ell_2(\Z) \to \ell_2(\Z).
$$
Since $\|B_w(x)\| \geq \beta \|x\|$ for all $x \in \ell_2(\Z)$,
we have that $B_w$ is uniformly positively expansive.
Since $B_w$ is invertible, $B_w$ is also uniformly expansive.
By \cite[Corollary 1]{FBayEMat09}, $B_w$ is supercyclic if and only if,
for any $q \in \N$,
$$
\liminf_{n \to \infty} \frac{w_0 \cdot\ldots\cdot w_{-n+q+1}}
                            {w_1 \cdot\ldots\cdot w_{n+q}}    = 0.
$$
But, by our choice of $w$,
$$
\liminf_{n \to \infty} \frac{w_0 \cdot\ldots\cdot w_{-n+q+1}}
                            {w_1 \cdot\ldots\cdot w_{n+q}}
  = \lim_{n \to \infty} \frac{\beta^{n-q}}{\alpha^{n+q}}
  = \frac{1}{\alpha^q \beta^q}
    \lim_{n \to \infty} \big( \frac{\beta}{\alpha}\big)^n
  = 0.
$$
Thus, the operator $B_w$ is also supercyclic.
\end{example}


\section{Expansivity and spectrum}

\shs We denote by $\D$ the open unit disc in the complex plane $\C$.
Moreover, $\rho(T)$ denotes the {\em resolvent set} of the operator $T$.
In the case $T$ is an operator on a real Banach space, we define
$$
\rho(T) = \rho(T_\C), \ \ \ \sigma_p(T) = \sigma_p(T_\C) \ \ \text{ and } \ \
\sigma_a(T) = \sigma_a(T_\C).
$$

It is known that if $T$ is a self-adjoint operator on a
complex Hilbert space $H$ and $x \in H$, then there exists a unique positive
Radon measure $\mu$ on $\sigma(T)$ such that
$$
\langle f(T)x,x \rangle = \int_{\sigma(T)} f(t) d\mu(t) \ \
  \text{ for all } f \in C(\sigma(T)).
$$
In particular, $\mu(\sigma(T)) = \|x\|^2$. The measure $\mu$ is called the
{\em spectral measure} associated to $T$ and $x$.

\smallskip
We refer the reader to the books \cite{HDow78} and \cite{Y} for more
informations concerning spectrum.

\medskip
\noindent {\bf Theorem C.}
{\it If $T$ is an operator on a Banach space $X$, then
$$
T \text{ uniformly positively expansive} \ \ \Rightarrow \ \
\sigma_a(T) \cap \ov{\D} = \emptyset.
$$
Moreover, the converse holds if $\rho(T) \cap \D \neq \emptyset$.
In particular, the converse holds if $T$ is invertible.}

\medskip
\begin{proof}
It is enough to consider the case of complex scalars.
Suppose that there is a point $\lambda \in \sigma_a(T) \cap
\ov{\D}$. Let $(x_k)_{k \in \N}$ be a sequence in $S_X$ such that
$\lim_{k \to \infty} \|\lambda x_k - Tx_k\| = 0$. Since
$$
\|\lambda^n x_k - T^nx_k\| \leq |\lambda| \|\lambda^{n-1} x_k -
T^{n-1} x_k\|
  + \|T^{n-1}\| \|\lambda x_k - Tx_k\|,
$$
it follows by induction that
$$
\lim_{k \to \infty} \|\lambda^n x_k - T^nx_k\| = 0 \ \ \text{ for
all }
  n \in \N.
$$
Since $\|\lambda^n x_k\| \leq 1$ for all $k,n \in \N$, we conclude
from Proposition~\ref{Characterizations}(b) that $T$ is not uniformly
positively expansive.

Now, assume that $\rho(T) \cap \D \neq \emptyset$ and $\sigma_a(T) \cap
\ov{\D} = \emptyset$. Since $\sigma_a(T)$ contains the boundary of
$\sigma(T)$, we must have $\sigma(T) \cap \ov{\D} = \emptyset$.
Hence,
$$
\sigma(T^{-1}) = \{\lambda^{-1} : \lambda \in \sigma(T)\} \subset \D,
$$
that is, $r(T^{-1}) < 1$. Choose $R \in \R$ such that $r(T^{-1}) < R < 1$.
It follows from the spectral radius formula that there exists $n_0 \in \N$
such that
$$
\|T^nx\| \geq R^{-n} \|x\| \ \ \text{ for all } x \in X \text{ and }
n \geq n_0,
$$
which implies that $T$ is uniformly positively expansive.
\end{proof}

Let us now give a short direct proof of the following result from \cite{MMaz00}.

\begin{theorem}\label{NormalExp}
If $T$ is an invertible normal operator on a Hilbert space $H$, then
$T$ is expansive if and only if $\sigma_p(T^*T) \cap \T = \emptyset$.
\end{theorem}

\begin{proof}
We may assume complex scalars.
Suppose that $\sigma_p(T^*T) \cap \T \neq \emptyset$ and let $\lambda$ be
a point in this intersection. There exists $x \in H \backslash \{0\}$
such that $T^*Tx = \lambda x$. Hence, for every $n \in \Z$,
$\|T^nx\|^2 = \langle (T^*T)^nx,x \rangle = \lambda^n \|x\|^2$,
implying that $\|T^nx\| = \|x\|$. Thus, $T$ is not expansive.

Conversely, assume that $T$ is not expansive and consider the positive
operator $S = T^*T$. There exists $x \in S_H$ with $\|T^nx\| < 2$
for all $n \in \Z$. Since $T$ is normal,
$$
\|S^nx\| = \|(T^n)^*T^nx\| = \|T^{2n}x\| < 2 \ \ \text{ for all } n \in \Z.
$$
Let $\mu$ be the spectral measure associated to $S$ and $x$. Since $S$
is an invertible positive operator, $\sigma(S) \subset (0,\infty)$.
Thus, by the Cauchy-Schwartz inequality,
$$
0 \leq \int_{\sigma(S)} t^n d\mu(t) = \langle S^nx,x \rangle
  \leq \|S^nx\|\|x\| < 2 \ \ \text{ for all } n \in \Z.
$$
For each $\alpha < 1$ and each $\beta > 1$, let
$A_\alpha = \sigma(S) \cap (0,\alpha]$ and
$B_\beta = \sigma(S) \cap [\beta,\infty)$.
Since
$$
\alpha^{-n} \mu(A_\alpha) \leq \int_{\sigma(S)} t^{-n}d\mu(t) < 2
\ \ \text{ and } \ \
\beta^n \mu(B_\beta) \leq \int_{\sigma(S)} t^n d\mu(t) < 2,
$$
for all $n \in \N$, we conclude that $\mu(A_\alpha) = \mu(B_\beta) = 0$.
This implies that $\sigma(S) \backslash \{1\}$ has $\mu$-measure zero.
Therefore,
$$
\|Sx - x\|^2 = \langle (S-I)^2x,x \rangle = \int_{\sigma(S)} (t-1)^2 d\mu(t)
  = 0,
$$
and so $1 \in \sigma_p(S)$.
\end{proof}

Recall that a Hilbert space operator $T$ is said to be {\em hyponormal} if
$$
\|T^*x\| \leq \|Tx\| \ \ \text{ for all } x.
$$
It is natural to ask if the previous theorem can be generalized to hyponormal
operators. Let us now show that the implication
$$
\sigma_p(T^*T) \cap \T = \emptyset \ \ \Longrightarrow \ \ T \text{ expansive}
$$
holds for hyponormal weighted shifts, but it is not true in general, and that
the converse implication may fail even for hyponormal weighted shifts.

\begin{proposition}\label{HypExpShift}
Let $w = (w_n)_{n \in \Z}$ be a weight sequence with $\inf_{n \in \Z} |w_n|>0$
and consider the bilateral weighted forward shift
$$
F_w : \ell_2(\Z) \to \ell_2(\Z).
$$
Assume that $F_w$ is hyponormal. If $\sigma_p(F_w^*F_w) \cap \T = \emptyset$,
then $F_w$ is expansive.
\end{proposition}

\begin{proof}
Since $F_w(e_n) = w_n e_{n+1}$ and $F_w^*(e_n) = \ov{w_{n-1}} e_n$,
$$
F_w^*F_w(e_n) = |w_n|^2 e_n \ \ \ (n \in \Z),
$$
which implies that
\begin{equation}
\sigma_p(F_w^*F_w) \cap \T = \emptyset \ \ \Longleftrightarrow \ \
|w_n| \neq 1 \ \text{ for all } n \in \Z. \label{equa}
\end{equation}
Since $F_w$ is hyponormal, it is well-known that the sequence
$(|w_n|)_{n \in \Z}$ is increasing. Therefore,
$\sup_{n \in \N} |w_1 \cdot\ldots\cdot w_n| = \infty$ if $|w_0| > 1$, while
$\sup_{n \in \N} |w_{-n} \cdot\ldots\cdot w_{-1}|^{-1} = \infty$ if $|w_0|<1$.
Anyway, it follows from Theorem~B(a) that $F_w$ is expansive.
\end{proof}

\begin{remark}\label{HypExpRemark}
(a) We cannot remove the hyponormality hypothesis in
Proposition~\ref{HypExpShift}. To see this, it is enough to choose $w$
so that $|w_n| \neq 1$ for all $n \in \Z$,
$\sup_{n \in \N} |w_1 \cdot\ldots\cdot w_n| < \infty$ and
$\sup_{n \in \N} |w_{-n} \cdot\ldots\cdot w_{-1}|^{-1} < \infty$.
Then, $\sigma_p(F_w^*F_w) \cap \T = \emptyset$ (by (\ref{equa})), but
$F_w$ is not expansive (by Theorem~B(a)).

\smallskip
\noindent (b) The converse of the conclusion of Proposition~\ref{HypExpShift}
is not true in general. For instance, if
$w = (\ldots,\frac{1}{2},\frac{1}{2},\frac{1}{2},1,2,2,2,\ldots)$
then $F_w$ is hyponormal and uniformly expansive, but
$\sigma_p(F_w^*F_w) \cap \T \neq \emptyset$.
\end{remark}

\begin{remark}
Let $T$ be a normal operator on a complex Hilbert space $H$.
In view of the previous theorem, it is natural to make the following
question: {\it Is it true that $T$ is positively expansive if and only if
$\sigma_p(T^*T) \cap \ov{\D} = \emptyset$?}

The direct implication is true, since the relation $T^*Tx = \lambda x$,
with $\lambda \in \ov{\D}$ and $x \neq 0$, implies that
$\|T^nx\|^2 = \langle (T^*T)^nx,x \rangle = \lambda^n \|x\|^2 \leq \|x\|^2$
for all $n \in \N$, and so $T$ is not positively expansive.

However, the converse is not true in general, even if $T$ is invertible.
Indeed, let $T : L^2[0,1] \to L^2[0,1]$ be defined by
$$
(Tf)(t) = \frac{t+1}{2}\, f(t) \ \ \text{ for all } f \in L^2[0,1].
$$
It is not difficult to see that $T$ is invertible, self-adjoint,
not positively expansive, and $\sigma_p(T^*T) = \emptyset$.
\end{remark}

Nevertheless, we have the following characterization.

\begin{proposition}\label{NormalPE}
Let $T$ be a normal operator on a complex Hilbert space $H$. Then, $T$ is
positively expansive if and only if $\mu(\sigma(T^*T) \cap (1,\infty)) > 0$
for every spectral measure $\mu$ associated to $T^*T$.
\end{proposition}

\begin{proof}
Let $S = T^*T$. If $T$ is not positively expansive, then there exists
$x \in S_H$ such that $\|T^nx\| < 2$ for all $n \in \N$. By letting $\mu$ be
the spectral measure associated to $S$ and $x$, we obtain
$$
0 \leq \int_{\sigma(S)} t^n d\mu(t) = \langle S^nx,x \rangle < 2
\ \ \text{ for all } n \in \N,
$$
which implies that $\mu(\sigma(S) \cap (1,\infty)) = 0$.

Conversely, suppose that for some $x \neq 0$, the spectral measure $\mu$
associated to $S$ and $x$ satisfies $\mu(\sigma(S) \cap (1,\infty)) = 0$.
Then,
$$
\|T^{2n}x\|^2 = \|S^nx\|^2 = \langle S^{2n}x,x \rangle
  = \int_{\sigma(S)} t^{2n} d\mu(t) \leq \|x\|^2
    \ \ \text{ for all } n \in \N,
$$
implying that $T$ is not positively expansive.
\end{proof}


\section{Shadowing}

\shs Given a metric space $M$ and a homeomorphism $h : M \to M$, recall that
a sequence $(x_n)_{n \in \Z}$ is a {\em $\delta$-pseudotrajectory} of $h$
($\delta > 0$) if
$$
d(h(x_n),x_{n+1}) \leq \delta \ \ \text{ for all } n \in \Z.
$$
The homeomorphism $h$ has the {\em shadowing property} \cite{
RBow75b}
if for every $\eps > 0$ there exists $\delta > 0$ such that every
$\delta$-pseudotrajectory $(x_n)_{n \in \Z}$ of $h$ is {\em $\eps$-shadowed}
by a real trajectory of $h$, that is, there exists $x \in M$ such that
$$
d(x_n,h^n(x)) < \eps \ \ \text{ for all } n \in \Z.
$$
Moreover, the homeomorphism $h$ has the {\em Lipschitz shadowing property} if there exists $K>0$ such that $\delta$ can be choosen satisfying that $\eps< K\delta$. More generally, we call it {\em $\alpha$-H\"older shadowing property}, $ 0 < \alpha \le 1$,  if $\delta$ can be chosen so that $ \eps <  K \delta^{\alpha}$. 
\begin{remark}\label{Shad1}
In the case of operators, it is enough to check the above condition for a
single $\eps > 0$. More precisely, if $T$ is an invertible operator on a
Banach space $X$ and if for some $\eps > 0$ there exists $\delta > 0$
such that every $\delta$-pseudotrajectory of $T$ is $\eps$-shadowed by
a real trajectory of $T$, then $T$ has the shadowing property. It is also true that any linear operator satisfying the shadowing property trivially satisfies the Lipschitz shadowing property.
These facts follows easily from the linearity of $T$.
\end{remark}

Many variations of the notion of shadowing have been introduced and studied
by several authors, e.g., \cite{SPil99,SPil12,WuOprChe16}.
We shall consider here the notions of limit shadowing and $\ell_p$ shadowing.
Let $M$ and $h$ be as above. The homeomorphism $h$ is said to have the
{\em limit shadowing property} if for every sequence $(x_n)_{n \in \Z}$
in $M$ with
$$
\lim_{|n| \to \infty} d(h(x_n),x_{n+1}) = 0,
$$
there exists $x \in M$ such that
$$
\lim_{|n| \to \infty} d(x_n,h^n(x)) = 0.
$$
Moreover, $h$ is said to have the {\em $\ell_p$ shadowing property}
($1 \leq p < \infty$) if for every sequence $(x_n)_{n \in \Z}$ in $M$ with
$$
\sum_{n \in \Z} d(h(x_n),x_{n+1})^p < \infty,
$$
there exists $x \in M$ such that
$$
\sum_{n \in \Z} d(x_n,h^n(x))^p < \infty.
$$

\begin{proposition}\label{Shad2}
Let $T$ be an invertible operator on a Banach space $X$. Suppose that
$$
X = M \oplus N,
$$
where $M$ and $N$ are closed $T$-invariant subspaces of $X$. Then $T$ has
the shadowing property (the limit shadowing property, the $\ell_p$ shadowing
property) if and only if so do $T|_M$ and $T|_N$.
\end{proposition}

\begin{proof}
By the open mapping theorem, there is a constant $\beta > 0$ such that
$$
\|a\| \leq \beta \|x\| \ \ \ \ \text{ and } \ \ \ \ \|b\| \leq \beta \|x\|,
$$
whenever $x = a + b$ with $a \in M$ and $b \in N$.
Let $x = a + b$ and $x_n = a_n + b_n$ ($n \in \Z$), where $a,a_n \in M$
and $b,b_n \in N$. The direct implication follows easily from the
inequalities
$$
\|a_n - (T|_M)^n(a)\| \leq \beta \|a_n - T^nx\| \ \ \text{ and } \ \
\|b_n - (T|_N)^n(b)\| \leq \beta \|b_n - T^nx\|,
$$
whereas the inverse implication follows easily from the inequalities
$$
\|(T|_M)(a_n) - a_{n+1}\| \leq \beta \|Tx_n - x_{n+1}\|, \ \
\|(T|_N)(b_n) - b_{n+1}\| \leq \beta \|Tx_n - x_{n+1}\|
$$
and
$$
\|x_n - T^nx\| \leq \|a_n - (T|_M)(a)\| + \|b_n - (T|_N)(b)\|.
$$
\end{proof}

\begin{corollary}\label{Shad3}
If $T$ is an invertible operator on a real Banach space $X$, then
$T$ has the shadowing property (the limit shadowing property, the $\ell_p$
shadowing property) if and only if so does its complexification $T_\C$.
\end{corollary}

In view of the above corollary, we will tacitly assume complex scalars
in all the proofs that follow.

\smallskip
It is well-known that any invertible hyperbolic operator on any Banach space
has the shadowing property (see \cite{JOmb94}, for instance). We shall
prove that it also has the limit shadowing property and the $\ell_p$
shadowing property for every $1 \leq p < \infty$. In fact, we will derive
this from a more general theorem which will also imply the existence of
nonhyperbolic operators with these shadowing properties. For this purpose,
we shall need two lemmas.

\begin{lemma} (see \cite{SPil99})
\label{LemmaA}
An invertible operator $T$ on a Banach space $X$ has the shadowing property
if and only if there is a constant $K > 0$ such that for every bounded
sequence $(z_n)_{n \in \Z}$ in $X$, there is a sequence $(y_n)_{n \in \Z}$
in $X$ such that
\begin{equation}
\sup_{n \in \Z} \|y_n\| \leq K \sup_{n \in \Z} \|z_n\|\label{Sha2.1}
\end{equation}
and
\begin{equation}
y_{n+1} = Ty_n + z_n \ \ \text{ for all } n \in \Z.\label{Sha3.1}
\end{equation}
\end{lemma}

\begin{proof}
Assume that $T$ has the shadowing property and let $\delta > 0$ be the
constant that appears in the definition of shadowing associated to $\eps = 1$.
Consider a bounded sequence $(z_n)_{n \in \Z}$ and put
$L = \sup_{n \in \Z} \|z_n\|$. Let $(x_n)_{n \in \Z}$ be such that
$x_{n+1} = Tx_n + \frac{\delta}{L} z_n$ for all
$n \in \Z$. Observe that $(x_n)_{n \in \Z}$ is completely determined by $x_0$.
Then $(x_n)_{n \in \Z}$ is a $\delta$-pseudotrajectory of $T$.
By hypothesis, there exists $x \in X$ such that
$$
\|x_n - T^nx\| < 1 \ \ \text{ for all } n \in \Z.
$$
By putting $y_n = \frac{L}{\delta}(x_n - T^nx)$, we have that (\ref{Sha2.1})
holds with $K = 1/\delta$ and (\ref{Sha3.1}) also holds.

For the converse, it is enough to consider $\eps = 1$
(Remark~\ref{Shad1}). Put $\delta = \frac{1}{2K}$ and let
$(x_n)_{n \in \Z}$ be a $\delta$-pseudotrajectory of $T$.
Put $z_n = x_{n+1} - Tx_n$ for all $n \in \Z$. By hypothesis, there exists
$(y_n)_{n \in \Z}$ such that (\ref{Sha2.1}) and (\ref{Sha3.1}) hold.
Since $x_{n+1} - y_{n+1} = T(x_n - y_n)$ for all $n \in \Z$, it follows that
$x_n - y_n = T^n(x_0 - y_0)$ for all $n \in \Z$. Thus,
$$
\|x_n - T^n(x_0 - y_0)\| = \|y_n\|
  \leq K \sup_{j \in \Z} \|z_j\| \leq K\delta = 1,
$$
for all $n \in \Z$.
\end{proof}

\begin{lemma}\label{LemmaB}
An invertible operator $T$ on a Banach space $X$ has the limit shadowing
property (the $\ell_p$ shadowing property) if and only if for every
sequence $(z_n)_{n \in \Z}$ in $X$ with
\begin{equation}
\lim_{|n| \to \infty} \|z_n\| = 0 \ \ \ \ \
\Big(\sum_{n \in \Z} \|z_n\|^p < \infty \Big),\label{Sha1}
\end{equation}
there exists a sequence $(y_n)_{n \in \Z}$ in $X$ such that
\begin{equation}
\lim_{|n| \to \infty} \|y_n\| = 0 \ \ \ \ \
\Big(\sum_{n \in \Z} \|y_n\|^p < \infty \Big)\label{Sha2}
\end{equation}
and
\begin{equation}
y_{n+1} = Ty_n + z_n \ \ \text{ for all } n \in \Z.\label{Sha3}
\end{equation}
\end{lemma}

\begin{proof}
Assume that $T$ has the limit shadowing property (the $\ell_p$ shadowing
property) and consider a sequence $(z_n)_{n \in \Z}$ satisfying (\ref{Sha1}).
Let $(x_n)_{n \in \Z}$ be such that $x_{n+1} = Tx_n + z_n$ for all $n \in \Z$.
Then, by hypothesis, there exists $x \in X$ such that
$\lim_{|n| \to \infty} \|x_n - T^nx\| = 0$
$\big(\sum_{n \in \Z} \|x_n - T^nx\|^p < \infty\big)$.
Hence, by putting $y_n = x_n - T^nx$ for all $n \in \Z$, we have that
(\ref{Sha2}) and (\ref{Sha3}) hold.

For the converse, consider $(x_n)_{n \in \Z}$ such that
$$
\lim_{|n| \to \infty} \|x_{n+1} - Tx_n\| = 0 \ \ \ \ \
\Big(\sum_{n \in \Z} \|x_{n+1} - Tx_n\|^p < \infty\Big).
$$
Put $z_n = x_{n+1} - Tx_n$ for all $n \in \Z$. Then $(z_n)_{n \in \Z}$
satisfies (\ref{Sha1}). Hence, by hypothesis, there exists
$(y_n)_{n \in \Z}$ such that (\ref{Sha2}) and (\ref{Sha3}) hold.
Since $x_n - T^n(x_0 - y_0) = y_n$ for all $n \in \Z$, we are done.
\end{proof}

\begin{theorem}\label{SufficienceShadowing}
Let $T$ be an invertible operator on a Banach space $X$. Suppose that
\begin{equation}
X = M \oplus N,\label{AA}
\end{equation}
where $M$ and $N$ are closed subspaces of $X$ with $T(M) \subset M$
and $T^{-1}(N) \subset N$. If
\begin{equation}
\sigma(T|_M) \subset \D \ \ \ \ \text{ and } \ \ \ \
\sigma(T^{-1}|_N) \subset \D,\label{BB}
\end{equation}
then
$T$ has the shadowing property, the limit shadowing property and the $\ell_p$
shadowing property
for all $1 \leq p < \infty$.
\end{theorem}

\begin{proof}
By (\ref{AA}), for each $x \in X$, there are unique $x^{(1)} \in M$ and
$x^{(2)} \in N$ satisfying $x = x^{(1)} + x^{(2)}$. Moreover, there is a
constant $\beta > 0$ such that
$$
\|x^{(1)}\| \leq \beta \|x\| \ \text{ and } \
\|x^{(2)}\| \leq \beta \|x\| \ \ \text{ for all } x \in X.
$$

By (\ref{BB}), $r(T|_M) < 1$ and $r(T^{-1}|_N) < 1$.
Choose $t \in \R$ such that
$$
\max\{r(T|_M),r(T^{-1}|_N)\} < t < 1.
$$
It follows from the spectral radius formula that there exists a constant
$C \geq 1$ such that
$$
\|(T|_M)^n\| \leq C\, t^n \ \text{ and } \
\|(T^{-1}|_N)^n\| \leq C\, t^n \ \ \text{ for all } n \in \N_0.
$$

Consider a bounded sequence $(z_n)_{n \in \Z}$ in $X$. For each $n \in \Z$,
we define
$$
y^{(1)}_n = \sum_{k=0}^\infty T^kz^{(1)}_{n-k-1} \in M, \ \ \
y^{(2)}_n = - \sum_{k=1}^\infty T^{-k}z^{(2)}_{n+k-1} \in N
$$
and
$$
y_n = y^{(1)}_n + y^{(2)}_n \in X.
$$
An easy computation shows that (\ref{Sha3.1}) (which is the same as
(\ref{Sha3})) holds. Moreover,
\begin{equation}
\|y^{(1)}_n\| \leq C \sum_{k=0}^\infty t^k \|z^{(1)}_{n-k-1}\|
\ \ \text{ and } \ \
\|y^{(2)}_n\| \leq C \sum_{k=1}^\infty t^k \|z^{(2)}_{n+k-1}\|.\label{Est}
\end{equation}
Hence,
$$
\sup_{n \in \Z} \|y_n\| \leq \Big(\frac{2 \beta C}{1-t}\Big)
\sup_{n \in \Z} \|z_n\|,
$$
which proves that $T$ has the shadowing property by Lemma~\ref{LemmaA}.

Now, assume that $(z_n)_{n \in \Z}$ satisfies (\ref{Sha1}).
By Lemma~\ref{LemmaB}, it remains to show that (\ref{Sha2}) holds.
By (\ref{Est}), for each $j \in \N$ and each $i \in \{1,2\}$,
$$
\|y^{(i)}_n\| \leq C \Bigg(\sum_{k=0}^j t^k\Bigg)
               \Bigg(\sup_{0 \leq k \leq j} \|z^{(i)}_{n+(-1)^ik-1}\|\Bigg)
           + C \Bigg(\sum_{k=j+1}^\infty t^k\Bigg)
               \Bigg(\sup_{k \in \Z} \|z^{(i)}_k\|\Bigg),
$$
which shows that $\lim_{|n| \to \infty} \|y_n\| = 0$ whenever
$\lim_{|n| \to \infty} \|z_n\| = 0$.
In the case $p = 1$, it is enough to use the estimates
$$
\sum_{n \in \Z} \|y^{(i)}_n\|
  \leq C \sum_{n \in \Z} \sum_{k=0}^\infty t^k \|z^{(i)}_{n+(-1)^ik-1}\|
     = C \Bigg(\sum_{k=0}^\infty t^k\Bigg)
         \Bigg(\sum_{n \in \Z} \|z^{(i)}_n\|\Bigg) \ \ \ \
  (i \in \{1,2\}).
$$
Finally, in the case $1 < p < \infty$, we consider its conjugate exponent
$q$ (i.e., $1/p + 1/q = 1$) and apply H\"older's inequality to obtain
$$
\|y^{(i)}_n\|
  \leq C \Bigg(\sum_{k=0}^\infty t^\frac{qk}{2}\Bigg)^\frac{1}{q}
       \Bigg(\sum_{k=0}^\infty t^\frac{pk}{2}
             \|z^{(i)}_{n+(-1)^ik-1}\|^p\Bigg)^\frac{1}{p} \ \ \ \
  (i \in \{1,2\}).
$$
As a consequence,
$$
\sum_{n \in \Z} \|y^{(i)}_n\|^p
  \leq C^p \Bigg(\sum_{k=0}^\infty t^\frac{qk}{2}\Bigg)^\frac{p}{q}
           \Bigg(\sum_{k=0}^\infty t^\frac{pk}{2}\Bigg)
           \Bigg(\sum_{n \in \Z} \|z^{(i)}_n\|^p\Bigg) \ \ \ \
  (i \in \{1,2\}).
$$
Thus, in all cases, (\ref{Sha2}) holds.
\end{proof}

As an immediate consequence, we have the following:

\begin{corollary}\label{HS}
Every invertible hyperbolic operator $T$ on a Banach space $X$ has the
shadowing property, the limit shadowing property and the $\ell_p$ shadowing
property for all $1 \leq p < \infty$.
\end{corollary}

It is well-known that the shadowing property implies hyperbolicity
in the cases of invertible operators on finite-dimensional euclidean spaces
\cite{JOmb94} and invertible normal operators on Hilbert spaces \cite{MMaz00}.
It is a basic question in linear dynamics whether or not this implication
is always true, that is, whether or not shadowing and hyperbolicity coincide
for invertible operators on Banach (or Hilbert) spaces.
This question appeared explicitly in \cite[Page 148]{MMaz00}, for instance.
Let us now answer this question in the negative as an application of
Theorem~\ref{SufficienceShadowing}.
The following is much stronger.

\medskip
\noindent {\bf Theorem D.}
{\it Let $X = \ell_q(\Z)$ $(1 \leq q < \infty)$ or $X = c_0(\Z)$.
There exists an invertible weighted shift $T$ on $X$ which satisfies the
frequent hypercyclicity criterion (hence it is not hyperbolic) and has the
shadowing property, the limit shadowing property and the $\ell_p$ shadowing
property for all $1 \leq p < \infty$.  Moreover any weighted shift operator sufficiently close to $T$ also satisfies the thesis of previous statement }

\medskip
Recall that an operator $T$ on a separable Banach space $X$ is said to
satisfy the {\em frequent hypercyclicity criterion} if there exist a dense
subset $X_0$ of $X$ and a map $S : X_0 \to X_0$ such that the following
properties hold for every $x \in X_0$:
\begin{itemize}
\item $\displaystyle \sum_{n=0}^\infty T^nx$ converges unconditionally;
\item $\displaystyle \sum_{n=0}^\infty S^nx$ converges unconditionally;
\item $TSx = x$.
\end{itemize}
If $T$ satisfies this criterion, then $T$ is frequently hypercyclic,
Devaney chaotic, mixing and densely distributionally chaotic; see
\cite[Section~9.2]{KGroAPer11} and \cite[Corollary~20]{BerBonMulPer13}.
Let us also recall that $T$ is said to be {\em Devaney chaotic} if it is
hypercyclic and has a dense set of periodic points. Of course, an invertible
operator which has a nontrivial periodic point is not expansive and, in
particular, it is not hyperbolic. Thus, the above theorem implies the
existence of operators with the shadowing property that are not hyperbolic
and exhibit several types of chaotic behaviors.

\noindent{\em Proof of theorem D}
Fix a real number $\alpha > 1$ and let $T$ be the bilateral weighted forward
shift on $X$ whose weight sequence $(w_n)_{n \in \Z}$ is given by
$w_n = \alpha$ if $n < 0$ and $w_n = 1/\alpha$ if $n \geq 0$.
By applying Theorem~\ref{SufficienceShadowing} with
$$
M = \{(x_n)_{n \in \Z} : x_n = 0 \text{ for all } n < 0\}
$$
and
$$
N = \{(x_n)_{n \in \Z} : x_n = 0 \text{ for all } n \geq 0\},
$$
we see that $T$ has all the above-mentioned shadowing properties.
Moreover, in order to see that $T$ satisfies the frequent hypercyclicity
criterion, it is enough to consider $X_0$ as the set of all sequences
$(x_n)_{n \in \Z}$ with finite support and $S = T^{-1}$. To conclude the second part of the thesis, observe that any weighted shift closed to $T$   also fits in Theorem \ref{SufficienceShadowing} (with the same $M$ and $N$) and satisfies the frequent hypercyclicity criterion.

\qed

\medskip

 The next   remarks highlight the differences between nonlinear finite dimensional dynamics and infinite dimensional linear dynamics,  explaining the status of the shadowing property for finite dimensional diffeomorphisms and raising a series of questions.
 
\begin{remark} As commented in the introduction, for $C^1$ diffeomorphisms on a finite dimensional manifold (see \cite{SPilSTik10}),  Lipschitz shadowing is equivalent to hyperbolicity. Our example proves that this is not the case for  infinite dimensional linear dynamics. In some sense, this shows that when one considers infinite dimensional spaces, even linear dynamics is richer than nonlinear finite dimensional dynamics.
\end{remark}

\begin{remark} The previous theorem  resembles a result in \cite{KP} where the existence of a nonhyperbolic yet having the shadowing property
$C^\infty$ diffeomorphism on a surface is exhibited. In view of the results in \cite{SPilSTik10},  the shadowing property can not be Lipschitz shadowing. Indeed, in  examples provided in \cite{KP} the shadowing property is only $\alpha$-H\"older for some $\alpha <1$.

\end{remark}

\begin{remark} It is worth pointing out  that finite dimensional diffeomorphisms induce infinite dimensional operators: for any  diffeomorphism one obtains finite dimensional linear cocycles  provided by the derivative of that diffeomorphism and those linear cocycles  can be recast as an infinite dimensional linear map (see for instances \cite{bochi} for discussions of linear cocycles). In particular, the proof in   \cite{SPilSTik10}
 is based on analyzing the dynamic of a diffeomorphisms as a linear cocycle, showing that the Lipschitz shadowing implies  shadowing for the cocycle of linear maps and from there concluding hyperbolicity using the results in \cite{VPli77}.   On the other hand, the spectrum problem related to certain nonlinear  infinite dimensional operators, as the  discrete Schr\"odinger operator can be reduced to a linear cocycle (see  \cite{damanik} for instance).
\end{remark}

\begin{remark} It is shown in theorem D that the shadowing property is satisfied for an open set of weighted shifts. It is natural to wonder if the same holds when one considers open sets of linear operators; in particular, is the shadowing property satisfied for any linear operator close to the ones that satisfies theorem D?
\end{remark}

As we have just seen, an invertible bilateral weighted shift can have
the shadowing property without being expansive. The next result presents
an additional condition which guarantees expansivity.

\begin{proposition}
Let $X = \ell_p(\Z)$ $(1 \leq p < \infty)$ or $X = c_0(\Z)$, and consider a
weight sequence $w = (w_n)_{n \in \Z}$ with $\inf_{n \in \Z} |w_n| > 0$.
If $F_w : X \to X$ has the shadowing property and the sequence
$(n F_w^n(e_0))_{n \in \Z}$ is not bounded, then $F_w$ is expansive.
\end{proposition}

We will prove at the end of this section that every expansive operator
with the shadowing property is uniformly expansive. So, in the above
proposition we can actually conclude that $F_w$ is uniformly expansive.

\begin{proof}
Assume that $F_w$ is not expansive. By Theorem~B(a),
$$
\sup_{n \in \N} |w_1 \cdot\ldots\cdot w_n| < \infty \ \ \text{ and } \ \
\sup_{n \in \N} |w_{-n} \cdot\ldots\cdot w_{-1}|^{-1} < \infty.
$$
Thus, the sequence $(z_n)_{n \in \Z}$ given by $z_n = F_w^{n+1}(e_0)$
($n \in \Z$) is bounded.
Since $F_w$ has the shadowing property, Lemma~\ref{LemmaA} guarantees the
existence of a bounded sequence $(y_n)_{n \in \Z}$ in $X$ such that
$$
y_{n+1} = F_w(y_n) + z_n \ \ \text{ for all } n \in \Z.
$$
For each $n \in \N$, note that
\begin{align*}
y_n &= F_w^n(y_0) + F_w^{n-1}(z_0) + \cdots + F_w(z_{n-2}) + z_{n-1}\\
    &= F_w^n(y_0) + n F_w^n(e_0)
\end{align*}
and
\begin{align*}
y_{-n} &= F_w^{-n}(y_0) - F_w^{-n}(z_{-1}) - F_w^{-n+1}(z_{-2})
          - \ldots - F_w^{-1}(z_{-n})\\
       &= F_w^{-n}(y_0) - n F_w^{-n}(e_0).
\end{align*}
Write $y_0 = (a_n)_{n \in \Z}$. Then
$$
\|y_n\| \geq |a_0 + n| |w_0 \cdot\ldots\cdot w_{n-1}| \ \ \text{ and } \ \
\|y_{-n}\| \geq |a_0 - n| |w_{-n} \cdot\ldots\cdot w_{-1}|^{-1},
$$
for every $n \in \N$. Since we are assuming that the sequence
$(n F_w^n(e_0))_{n \in \Z}$ is not bounded, these estimates imply that
the sequence $(y_n)_{n \in \Z}$ is not bounded, a contradiction.
\end{proof}

Let us now see that all notions of shadowing considered here coincide
in the finite dimensional setting. As mentioned before the equivalence 
(i) $\Leftrightarrow$ (ii) below is already known (see \cite{JOmb94}, where further references can be found).

\begin{proposition}\label{FD1}
Fix $p \in [1,\infty)$. If $T$ is an invertible operator on a
finite-dimensional Banach space $X$, then the following assertions are
equivalent:
\begin{enumerate}
\item [\rm (i)]   $T$ is hyperbolic;
\item [\rm (ii)]  $T$ has the shadowing property;
\item [\rm (iii)] $T$ has the limit shadowing property;
\item [\rm (iv)]  $T$ has the $\ell_p$ shadowing property.
\end{enumerate}
\end{proposition}

\begin{proof}
Suppose that $T$ is not hyperbolic.
Since it is easy to see that if  $T$ has the
shadowing property (the limit shadowing property, the $\ell_p$ shadowing
property), then so does $\lambda T$ whenever $|\lambda| = 1$, we may assume that $1 \in \sigma(T)$.
Hence, by Proposition~\ref{Shad2} and the Jordan canonical form, it is
enough to consider the cases in which $T$ is an operator on $\C^k$ (with
the euclidean norm) whose canonical matrix has the form
\begin{equation}
\big[1\big] \ \ \ \ \text{ or } \ \ \ \
\left[
\begin{array}{ccccc}
1 &         &        &   \\
1 &    1    &        &   \\
  & \ddots  & \ddots &   \\
  &         &    1   & 1 \\
\end{array}
\right].\label{Sha4}
\end{equation}
We shall handle both cases simultaneously, but we remark that in the first
case the vectors we are going to define have only the first coordinate.

Assume $p > 1$ and consider the sequence $(x_n)_{n \in \Z}$ in $\C^k$
given by $x_n = (0,\ldots,0)$ for $n \leq 0$ and
$$
x_{n+1} = \Big(1 + \frac{1}{2} + \cdots + \frac{1}{n+1},
               (Tx_n)_2,\ldots,(Tx_n)_k\Big) \ \ \text{ for } n \geq 0,
$$
where $(x)_j$ denotes the $j^{\text{th}}$ coordinate of the vector
$x \in \C^k$. Then
$$
\sum_{n \in \Z} \|Tx_n - x_{n+1}\|^p
  = \sum_{n=0}^\infty \Big(\frac{1}{n+1}\Big)^p
  < \infty.
$$
However, it is not possible to find an $x \in X$ such that
$$
\lim_{|n| \to \infty} \|x_n - T^n x\| = 0,
$$
since
$$
\|x_n - T^n x\| \geq |(x_n - T^n x)_1| = \Big|\Big(1 + \frac{1}{2} + \cdots
  + \frac{1}{n}\Big) - (x)_1\Big| \to \infty \ \ \text{ as } n \to \infty,
$$
for every $x \in \C^k$. This shows that $T$ does not have the limit shadowing
property nor the $\ell_p$ shadowing property for any $p > 1$.

Now, assume $p = 1$. Consider the sequence $(x_n)_{n \in \Z}$ in $\C^k$
given by $x_n = (0,\ldots,0)$ for $n \leq 0$ and
$$
x_{n+1} = \Big(\frac{1}{n+1},(Tx_n)_2,\ldots,(Tx_n)_k\Big) \ \
  \text{ for } n \geq 0.
$$
Then
$$
\sum_{n \in \Z} \|Tx_n - x_{n+1}\| = 1 + \sum_{n=1}^\infty \frac{1}{n(n+1)}
  < \infty.
$$
Nevertheless,
$$
\sum_{n \in \Z} \|x_n - T^n x\| \geq \sum_{n \in \Z} |(x_n - T^n x)_1|
  \ge \sum_{n = 1}^\infty \Big|\frac{1}{n} - (x)_1\Big| = \infty,
$$
for every $x \in \C^k$. Thus, $T$ does not have the $\ell_1$ shadowing
property.
\end{proof}

\smallskip
We say that a sequence $(t_n)_{n \in \Z}$ of scalars is $O(|n|)$ if there
exist $\alpha > 0$, $\beta > 0$ and $n_0 \in \N$ such that
$$
\alpha |n| \leq |t_n| \leq \beta |n| \ \ \text{ whenever } |n| \geq n_0.
$$

Let us now establish a result which will imply a much simpler and shorter
proof of the main result in \cite{MMaz00}
(see Corollary~\ref{NormalShadowing1}).

\begin{theorem}\label{ShadowingOn}
Let $T$ be an invertible operator on a Banach space $X$ such that for all
$z \in X$, the sequence $(\|T^nz\|)_{n \in \Z}$ is not $O(|n|)$. If $T$ has
the shadowing property, then $T$ is uniformly expansive.
\end{theorem}

\begin{proof}
Suppose that $T$ has the shadowing property and let $\delta > 0$ be the
constant that appears in the definition of shadowing associated to $\eps = 1$.
Assume that $T$ is not uniformly expansive. Then, by \cite[Theorem 1]{JHed71},
the intersection $\sigma_a(T) \cap \T$ is nonempty. Take a scalar $\lambda$
in this intersection. Hence, there exists $x_0 \in X$ such that
$$
\|x_0\| = 1 \ \ \ \text{ and } \ \ \ \|\lambda x_0 - Tx_0\| < \frac{\delta}{2}\cdot
$$
For each $n \in \Z$, let $y_n = 2\lambda^n x_0$. Then $(y_n)_{n \in \Z}$ is a
$\delta$-pseudotrajectory of $T$, and so there exists $y \in X$ such that
$\|y_n - T^ny\| < 1$ for all $n \in \Z$. Therefore,
$$
1 < \|T^ny\| < 3 \ \ \text{ for all } n \in \Z.
$$
Now, consider the sequence $(z_n)_{n \in \Z}$ defined by
$z_n = \frac{n \delta}{3} \, T^ny$ for all $n \in \Z$.
Since $(z_n)_{n \in \Z}$ is a $\delta$-pseudotrajectory of $T$, there exists
$z \in X$ such that $\|z_n - T^nz\| < 1$ for all $n \in \Z$. Thus,
$$
\frac{|n| \delta}{3} - 1 < \|T^n z\| < |n| \delta + 1 \ \
  \text{ for all } n \in \Z.
$$
This contradicts the fact that $(\|T^nz\|)_{n \in \Z}$ is not $O(|n|)$.
\end{proof}

\begin{corollary}\label{NormalShadowing1}
If $T$ is an invertible normal operator on a Hilbert space $H$, then
$T$ has the shadowing property if and only if $T$ is hyperbolic.
\end{corollary}

\begin{proof}
Suppose that $T$ has the shadowing property but is not hyperbolic.
Since $T$ is normal, $\sigma(T) = \sigma_a(T)$, and so $T$ is not uniformly
expansive. Hence, by Theorem~\ref{ShadowingOn}, there exists $z \in H$
such that $(\|T^nz\|)_{n \in \Z}$ is $O(|n|)$. Let $\alpha > 0$, $\beta > 0$
and $n_0 \in \N$ be such that
\begin{equation}
\alpha |n| \leq \|T^nz\| \leq \beta |n| \ \ \text{ whenever } |n| \geq n_0.
\label{On}
\end{equation}
Consider the invertible positive operator $S = T^*T$ and let $\mu$ be the
spectral measure associated to $S$ and $z$. Then,
$$
0 \leq \int_{\sigma(S)} t^n d\mu(t) = \langle S^nz,z \rangle = \|T^nz\|^2
  \leq \beta^2 |n|^2  \ \ \text{ whenever } |n| \geq n_0.
$$
By arguing as in the proof of Theorem~\ref{NormalExp} (with the sets
$A_\alpha$ and $B_\beta$), we see that $\sigma(S) \backslash \{1\}$ has
$\mu$-measure zero and so $Sz = z$. This implies that $\|T^nz\| = \|z\|$
for all $n \in \Z$, which contradicts the first inequality in (\ref{On}).
\end{proof}

The next proposition gives another additional condition under which shadowing
implies uniform expansivity (compare with Theorem~\ref{ShadowingOn}).

\begin{proposition}\label{ESPUE}
Let $T$ be an invertible operator on a Banach space $X$. If $T$ is
expansive and has the shadowing property, then $T$ is uniformly expansive.
\end{proposition}

\begin{proof}
Suppose that $T$ is not uniformly expansive and fix a scalar
$\lambda \in \sigma_a(T) \cap \T$. By arguing as in the proof of
Theorem~\ref{ShadowingOn}, we obtain a vector $y \in X$ such that
$$
1 < \|T^ny\| < 3 \ \ \text{ for all } n \in \Z.
$$
This contradicts the hypothesis that $T$ is expansive.
\end{proof}


\section{Positive shadowing}

\shs
In the case $h$ is a continuous self-map of a metric space $M$, we can define
the notion of {\em positive shadowing} simply by replacing the set $\Z$ by
the set $\N_0$ in the definition of shadowing. Similarly, we can define the
notions of {\em positive limit shadowing} and {\em positive $\ell_p$
shadowing} for such a map $h$.

\smallskip
Remark \ref{Shad1}, Proposition \ref{Shad2} and Corollary \ref{Shad3} have
analogous statements for not necessarily invertible
operators if we add the word ``positive'' to the corresponding notions
of shadowing.

\medskip
\noindent {\bf Theorem E.}
{\it Let $T$ be a (not necessarily invertible) operator on a Banach space $X$.
If $T$ is hyperbolic, then $T$ has the positive shadowing property, the
positive limit shadowing property and the positive $\ell_p$ shadowing
property for all $1 \leq p < \infty$.}

\medskip
\begin{proof}
We divide the proof in three cases.

\medskip
\noindent {\bf Case 1.} $\sigma(T) \subset \D$.

\medskip
Then there exist $t \in (0,1)$ and $C \geq 1$ such that
$$
\|T^n\| \leq C\, t^n \ \ \text{ for all } n \in \N_0.
$$

Given $\eps > 0$, put $\delta = \frac{(1-t)\eps}{C}\cdot$
Let $(x_n)_{n \in \N_0}$ be a $\delta$-pseudotrajectory of $T$ and define
$y_n = x_n - Tx_{n-1}$ for $n \in \N$. Then
\begin{equation}
x_n = T^nx_0 + T^{n-1}y_1 + T^{n-2}y_2 + \cdots + Ty_{n-1} + y_n \ \
  \text{ for all } n \in \N. \label{PSha1}
\end{equation}
Since $\|y_n\| \leq \delta$ for all $n \in \N$, we conclude that
$$
\|x_n - T^nx_0\| \leq Ct^{n-1}\delta + Ct^{n-2}\delta + \cdots + Ct\delta + \delta
  < \frac{C\delta}{1-t} = \eps \ \ \ (n \in \N).
$$
Hence, $(x_n)_{n \in \N_0}$ is $\eps$-shadowed by $(T^nx_0)_{n \in \N_0}$.
This proves that $T$ has the positive shadowing property.

Let $(x_n)_{n \in \N_0}$ be a sequence in $X$ with
$$
\lim_{n \to \infty} \|Tx_n - x_{n+1}\| = 0.
$$
Let $y_n$ be defined as above. By (\ref{PSha1}),
\begin{align*}
\|x_n - T^n x_0\|
  &\leq Ct^{n-1}\|y_1\| + Ct^{n-2}\|y_2\| + \cdots + Ct\|y_{n-1}\| + \|y_n\|\\
  &\leq C \Bigg(\sum_{k=0}^j t^k\Bigg)
          \Bigg(\sup_{0 \leq k \leq j} \|y_{n-k}\|\Bigg)
      + C \Bigg(\sum_{k=j+1}^{n-1} t^k\Bigg)
          \Bigg(\sup_{k \in \N} \|y_k\|\Bigg),
\end{align*}
whenever $0 < j < n$. Since $\|y_n\| \to 0$, the above estimate implies that
$\|x_n - T^nx_0\| \to 0$ as well. Thus, $T$ has the positive limit shadowing
property.

Now, suppose that
$$
\sum_{n=0}^\infty \|Tx_n - x_{n+1}\| < \infty.
$$
Then, by (\ref{PSha1}),
\begin{align*}
\sum_{n=0}^\infty \|x_n - T^n x_0\|
  &\leq C \sum_{n=1}^\infty \sum_{k=0}^{n-1} t^k \|y_{n-k}\|
      = C \sum_{k=0}^\infty \sum_{n=k+1}^\infty t^k \|y_{n-k}\|\\
  &=    C \Bigg(\sum_{k=0}^\infty t^k\Bigg)
          \Bigg(\sum_{n=1}^\infty \|y_n\|\Bigg) < \infty.
\end{align*}

Finally, suppose that $1 < p < \infty$ and that
$$
\sum_{n=0}^\infty \|Tx_n - x_{n+1}\|^p < \infty.
$$
By (\ref{PSha1}) and H\"older's inequality,
$$
\|x_n - T^n x_0\|
  \leq C \sum_{k=0}^{n-1} t^k\|y_{n-k}\|
  \leq C \Bigg(\sum_{k=0}^{n-1} t^\frac{qk}{2}\Bigg)^\frac{1}{q}
         \Bigg(\sum_{k=0}^{n-1} t^\frac{pk}{2} \|y_{n-k}\|^p\Bigg)^\frac{1}{p},
$$
where $q$ is the conjugate exponent to $p$. Thus,
\begin{align*}
\sum_{n=0}^\infty \|x_n - T^n x_0\|^p
  \leq    C^p \Bigg(\sum_{k=0}^\infty t^\frac{qk}{2}\Bigg)^\frac{p}{q}
            \Bigg(\sum_{k=0}^\infty t^\frac{pk}{2}\Bigg)
            \Bigg(\sum_{n=1}^\infty \|y_n\|^p\Bigg) < \infty.
\end{align*}
Therefore, $T$ also has the positive $\ell_p$ shadowing property.

\medskip
\noindent {\bf Case 2.} $\sigma(T) \subset \C \backslash \ov{\D}$.

\medskip
Then $T$ is invertible and we can apply Corollary~\ref{HS}.

\medskip
\noindent {\bf Case 3.} $\sigma(T) \cap \D \neq \emptyset$ and
$\sigma(T) \cap (\C \backslash \ov{\D}) \neq \emptyset$.

\medskip
In this case, the sets $\sigma_1 = \sigma(T) \cap \D$ and
$\sigma_2 = \sigma(T) \cap (\C \backslash \ov{\D})$ form a partition of
$\sigma(T)$ into two nonempty closed sets. By the Riesz decomposition
theorem \cite[Theorem~B.9]{KGroAPer11}, there are nontrivial $T$-invariant
closed subspaces $M_1$ and $M_2$ of $X$ such that
$$
X = M_1 \oplus M_2, \ \ \ \sigma(T|_{M_1}) = \sigma_1 \ \ \text{ and } \ \
\sigma(T|_{M_2}) = \sigma_2.
$$
By Cases 1 and 2, both $T|_{M_1}$ and $T|_{M_2}$ have the positive shadowing
property, the positive limit shadowing property and the positive $\ell_p$
shadowing property for all $1 \leq p < \infty$, from which it follows easily
that $T$ also has these properties.
\end{proof}

\begin{remark}\label{NotConverse}
The converse to Theorem~E is not true in general. Indeed, the operator
constructed in the proof of Theorem~D has the positive shadowing property,
the positive limit shadowing property and the positive $\ell_p$ shadowing
property for all $1 \leq p < \infty$, but it is not hyperbolic.
\end{remark}

Let us now prove that all notions of positive shadowing considered here
coincide with hyperbolicity in the case of compact operators. First, let us
consider the case of finite-dimensional spaces.

\begin{lemma}\label{FD2}
Fix $p \in [1,\infty)$. If $T$ is an operator on a finite-dimensional Banach
space $X$, then the following assertions are equivalent:
\begin{enumerate}
\item [\rm (i)]   $T$ is hyperbolic;
\item [\rm (ii)]  $T$ has the positive shadowing property;
\item [\rm (iii)] $T$ has the positive limit shadowing property;
\item [\rm (iv)]  $T$ has the positive $\ell_p$ shadowing property.
\end{enumerate}
\end{lemma}

\begin{proof}
Suppose that $T$ is not hyperbolic. We have to prove that (ii), (iii) and
(iv) are all false. By arguing as in the proof of Theorem~\ref{FD1}, we
see that it is enough to consider the cases in which $T$ is an operator on
$\C^k$ (with the euclidean norm) whose canonical matrix has one of the forms
in (\ref{Sha4}).

Fix $\delta > 0$ and let $(x_n)_{n \in \N_0}$ in $\C^k$ be given by
$x_0 = (0,\ldots,0)$ and
$$
x_{n+1} = \big((n+1)\delta,(Tx_n)_2,\ldots,(Tx_n)_k\big)
  \ \ \text{ for } n \in \N_0.
$$
Then $(x_n)_{n \in \N_0}$ is a $\delta$-pseudotrajectory of $T$ that cannot
be shadowed by the real trajectory of any point $x \in \C^k$, because
$$
\|x_n - T^nx\| \geq |(x_n - T^nx)_1| = |n\delta - (x)_1| \to \infty
  \ \ \text{ as } n \to \infty.
$$
So, $T$ does not have the positive shadowing property.

The proofs that $T$ does not have the positive limit shadowing property
and does not have the positive $\ell_p$ shadowing property are similar
to the corresponding proofs in Theorem~\ref{FD1} and so we omit them.
\end{proof}

\begin{theorem}\label{CompactShadowing}
Fix $p \in [1,\infty)$. If $T$ is a compact operator on a Banach space $X$,
then the following assertions are equivalent:
\begin{enumerate}
\item [\rm (i)]   $T$ is hyperbolic;
\item [\rm (ii)]  $T$ has the positive shadowing property;
\item [\rm (iii)] $T$ has the positive limit shadowing property;
\item [\rm (iv)]  $T$ has the positive $\ell_p$ shadowing property.
\end{enumerate}

\end{theorem}

\begin{proof}
Suppose that $T$ has the positive shadowing property (the positive limit
shadowing property, the positive $\ell_p$ shadowing property). We have to
prove that $T$ is hyperbolic.
We may assume that $X$ is infinite-dimensional (because of Lemma~\ref{FD2})
and that $\sigma(T)$ is not contained in $\D$.
Since $T$ is a compact operator, it follows that the sets
$\sigma_1 = \sigma(T) \cap \D$ and $\sigma_2 = \sigma(T) \backslash \D$
form a partition of $\sigma(T)$ into two nonempty closed sets.
By the Riesz decomposition theorem, there are nontrivial $T$-invariant
closed subspaces $M_1$ and $M_2$ of $X$ such that
$$
X = M_1 \oplus M_2, \ \ \ \sigma(T|_{M_1}) = \sigma_1 \ \ \text{ and } \ \
\sigma(T|_{M_2}) = \sigma_2.
$$
The compactness of $T$ also implies that $M_2$ is finite-dimensional.
Hence, since $T|_{M_2}$ has the positive shadowing property (the positive
limit shadowing property, the positive $\ell_p$ shadowing property),
Lemma~\ref{FD2} tell us that $\sigma_2 \cap \T = \emptyset$.
Thus, $\sigma(T) \cap \T = \emptyset$.
\end{proof}

Let us now show that the notions of hyperbolicity and positive shadowing
coincide for normal operators.

\begin{theorem}\label{NormalShadowing2}
If $T$ is a normal operator on a Hilbert space $H$, then
$T$ has the positive shadowing property if and only if $T$ is hyperbolic.
\end{theorem}

\begin{proof}
Suppose that $T$ has the positive shadowing property.
Assume that $T$ is not hyperbolic and argue as in the proof of
Theorem~\ref{ShadowingOn} to obtain a vector $z \in H$ such that
\begin{equation}
\frac{n \delta}{3} - 1 < \|T^nz\| < n \delta + 1 \ \
  \text{ for all } n \in \N_0. \label{B}
\end{equation}
Consider the positive operator $S = T^*T$ and let $\mu$ be the spectral
measure associated to $S$ and $z$. Since
$$
0 \leq \int_{\sigma(S)} t^n d\mu(t) = \langle S^nz,z \rangle = \|T^nz\|^2
  \leq (n\delta + 1)^2 \ \ \text{ for all } n \in \N_0,
$$
it follows that $\sigma(S) \cap (1,\infty)$ has $\mu$-measure zero.
Hence, $\|T^nz\| \leq (\mu(\sigma(S)))^\frac{1}{2}$ for all $n \in \N_0$,
which contradicts (\ref{B}).
\end{proof}


\smallskip
\noindent Nilson C. Bernardes Jr.\\
{\it Departamento de Matem\'atica Aplicada, Instituto de Matem\'atica,
Universidade Federal do Rio de Janeiro, Caixa Postal 68530, Rio de Janeiro,
RJ, 21945-970, Brasil.}\\
{\it e-mail:} bernardes@im.ufrj.br

\medskip
\noindent Patricia R. Cirilo\\
{\it Instituto de Ci\^encia e Tecnologia, Universidade Federal de S\~ao Paulo,
Avenida Cesare Mansueto Giulio Lattes, n$^{\rm o}$ 1201, S\~ao Jos\'e dos
Campos, SP, 12247-014, Brasil.}\\
{\it e-mail:} pcirilo@unifesp.br

\medskip
\noindent Udayan B. Darji\\
{\it Department of Mathematics, University of Louisville, Louisville,
KY 40292, USA.}\\
{\it Department of Mathematics, Ashoka University, Rajiv Gandhi Education City Kundli, Rai 131029, India}\\
{\it e-mail:} ubdarj01@louisville.edu

\medskip
\noindent Ali Messaoudi\\
{\it Departamento de Matem\'atica, Universidade Estadual Paulista,
Rua Crist\'ov\~ao Colombo, 2265, Jardim Nazareth, S\~ao Jos\'e do Rio Preto,
SP, 15054-000, Brasil.}\\
{\it e-mail:} messaoud@ibilce.unesp.br

\medskip
\noindent Enrique R. Pujals\\
{\it Instituto Nacional de Matem\'atica Pura e Aplicada, Estrada Dona
Castorina 110, Rio de Janeiro, RJ, 22460-320, Brasil.}\\
{\it e-mail:} enrique@impa.br

\end{document}